\documentclass[a4paper,12pt,leqno]{amsart}
\usepackage{amsmath,amsthm,amssymb}
\usepackage{mathdots}

\numberwithin{equation}{section}

\theoremstyle{plain}
\newtheorem{theorem}{Theorem}[section]
\newtheorem{proposition}[theorem]{Proposition}

\newtheorem{definition}[theorem]{Definition}

\theoremstyle{definition}

\newtheorem{remark}[theorem]{Remark}

\newcommand{\C}{\mathbb C}
\newcommand{\R}{\mathbb R}
\newcommand{\Z}{\mathbb Z}
\newcommand{\N}{\mathbb N}

\newcommand{\al}{\alpha}

\newcommand{\de}{\delta}

\newcommand{\si}{\sigma}

\newcommand{\eps}{\epsilon}
\newcommand{\Om}{\Omega}
\newcommand{\De}{\Delta}
\renewcommand{\th}{\theta}
\newcommand{\om}{\omega}

\newcommand{\ze}{\zeta}

\DeclareMathOperator{\diag}{diag}

\DeclareMathOperator{\Hom}{Hom}

\DeclareMathOperator{\Fix}{Fix}
\DeclareMathOperator{\id}{id}

\DeclareMathOperator{\Ad}{Ad}

\newcommand{\SL}{\textrm{SL}}
\newcommand{\SO}{\textrm{SO}}
\newcommand{\GL}{\textrm{GL}}
\renewcommand{\sl}{\frak s\frak l}

\newcommand{\m}{\mathfrak{m}}

\newcommand{\cR}{\mathcal{R}}
\newcommand{\cS}{\mathcal{S}}
\newcommand{\cM}{\mathcal{M}}

\newcommand{\cZ}{\mathcal{Z}}
\newcommand{\cT}{\mathcal{T}}
\newcommand{\cG}{\mathcal{G}}

\newcommand{\no}{\noindent}
\newcommand{\sub}{\subseteq}
\newcommand{\st}{\ \vert\ }
\renewcommand{\ll}{\lq\lq}
\newcommand{\rr}{\rq\rq\ }
\newcommand{\rrr}{\rq\rq}

\newcommand{\bp}{\begin{pmatrix}}
\newcommand{\ep}{\end{pmatrix}}
\newcommand{\bsp}{\left(\begin{smallmatrix}}
\newcommand{\esp}{\end{smallmatrix}\right)}

\newcommand{\ttb}{ {t\bar t}  }

\renewcommand{\i}{ {\scriptscriptstyle\sqrt{-1}}\, }
\newcommand{\ii}{ {\scriptstyle\sqrt{-1}}\, }

\newcommand{\Psiz}{  \Psi^{(0)}  }
\newcommand{\Psii}{  \Psi^{(\infty)}  }
\newcommand{\psiz}{  \psi^{(0)}  }
\newcommand{\psii}{  \psi^{(\infty)}  }
\newcommand{\Sz}{  S^{(0)}  }
\newcommand{\Si}{  S^{(\infty)}  }
\newcommand{\Qz}{  Q^{(0)}  }
\newcommand{\Qi}{  Q^{(\infty)}  }

\newcommand{\Omz}{  \Om^{(0)}  }
\newcommand{\Omi}{  \Om^{(\infty)}  }
\newcommand{\thz}{  \th^{(0)}  }

\newcommand{\Mz}{ M^{(0)}  }

\newcommand{\tPsiz}{  \tilde\Psi^{(0)}  }
\newcommand{\tPsii}{  \tilde\Psi^{(\infty)}  }

\newcommand{\tSz}{  \tilde S^{(0)}  }

\newcommand{\tQz}{  \tilde Q^{(0)}  }
\newcommand{\tQi}{  \tilde Q^{(\infty)}  }

\newcommand{\tE}{  \tilde E  }

\newcommand{\tMz}{  \tilde M^{(0)}  }

\newcommand{\tC}{  \tilde C }

\newcommand{\cRz}{ \cR^{(0)}}

\newcommand{\rra}{\rightrightarrows}

\newcommand{\bartQz}{  {\overline{\tilde Q}}^{(0)}  }

\newcommand{\bsi}{\boldsymbol\si}
\newcommand{\bth}{\boldsymbol\th}

\newcommand{\FA}{ F(A) }

\newcommand{\phn}{ \hat\Pi^{(n+1)/2}  }

\newcommand{\tQzr}{  \tilde Q^{(0),\text{rev}}  }
\newcommand{\tSzr}{  \tilde S^{(0),\text{rev}}  }

\usepackage{xcolor}

\newcommand{\cC}{\mathcal{C}}

\newcommand{\bC}{\mathbb{C}}

\newcommand{\fg}{\mathfrak{g}}
\newcommand{\fm}{\mathfrak{m}}

\newcommand{\pr}{\mathrm{pr}}

\newcommand{\Tr}{\mathrm{Tr}}

\newcommand{\Lie}{\mathrm{Lie}}
\newcommand{\reg}{\mathrm{reg}}
\newtheorem{thm}{Theorem}[section]

\newtheorem{lm}[thm]{Lemma}
\newtheorem{rem}[thm]{Remark}
\newtheorem{cor}[thm]{Corollary}
\newtheorem{pro}[thm]{Proposition}
\newtheorem{ex}[thm]{Example}
\newtheorem{df}[thm]{Definition}

\renewcommand{\sl}{\frak s\frak l}

\begin{document}

\title[tt*-Toda and symplectic groupoids]{Geometry of the tt*-Toda equations
\\
I: universal centralizer
\\
and symplectic groupoids
}

\author{Martin A. Guest}
\address{Department of Mathematics, Waseda University, 3-4-1 Okubo, Shinjuku, Tokyo 169-8555, Japan}
\email{martin@waseda.jp}

\author{Nan-Kuo Ho}
\address{Department of Mathematics, National Tsing Hua University, Hsinchu 300, and National Center for Theoretical Sciences, Taipei 106,  Taiwan}
\email{nankuo@math.nthu.edu.tw}

\date{}

\begin{abstract}
We investigate the geometry of a certain space of meromorphic connections with irregular singularities, and prove in particular that it is a (real) symplectic Lie groupoid.
The connections have a physical meaning: they correspond to certain solutions of the topological-antitopological fusion (tt*) equations of Cecotti and Vafa, and hence to deformations of supersymmetric quantum field theories. The groupoid structure
arises because we restrict ourselves to the tt* equations of Toda type, whose monodromy data has a Lie theoretic description. To obtain these results, we show first that the universal centralizer of a Lie group is a holomorphic symplectic groupoid over the Steinberg cross section.\end{abstract}


\maketitle

\section{Introduction}\label{intro}

We shall describe some geometric aspects of the space of isomonodromic data
of the tt*-Toda equations.  In this article we focus on the equations of $A_n$-type.
These are the system
\begin{equation}\label{ost}
 2(w_i)_{\ttb}=-e^{2(w_{i+1}-w_{i})} + e^{2(w_{i}-w_{i-1})}, \
 w_i:\C^\ast\to\R, \
 i\in\Z
\end{equation}
together with the conditions
$w_i=w_{i+n+1}$ (periodicity),
$w_i=w_i(\vert t\vert)$
(radial condition),
$w_i+w_{n-i}=0$
(\ll anti-symmetry\rrr).
They are a version of the 2D periodic Toda equations, a well known integrable system. They are also an
example of the 2D tt* equations (topological-antitopological fusion equations),  which were introduced by Cecotti and Vafa in \cite{CeVa91} and \cite{CeVa92a} to describe deformations of supersymmetric field theories.
This confluence of ideas makes the tt*-Toda equations a fruitful source of interaction between geometry and physics.

Let us review some of these ideas briefly.
First, it is well known (see \cite{GuLi14})
that the tt*-Toda equations (and indeed any 2D tt* equations) have a \ll zero curvature formulation\rrr, i.e.\ (\ref{ost}) is equivalent to the condition $d\al+\al\wedge\al=0$ for a certain connection $1$-form $\al$, and that this makes (\ref{ost})
also a special case of the Hitchin equations, i.e.\ the equations for a harmonic bundle --- locally, solutions are harmonic maps into the symmetric space $\SL_{n+1}\R/\SO_{n+1}$.
Corresponding to a solution, there is a
meromorphic
Higgs field, which, in our situation, can be described explicitly (see \cite{GIL3}).

Next, it was recognised by Dubrovin (see \cite{Du93}) that, thanks to the radial condition,  the tt*-Toda equations (and indeed any 2D tt* equations)  have, in addition, an \ll isomonodromy formulation\rrr, namely that
(\ref{ost}) is  equivalent to the condition that a certain {\em meromorphic} connection form $\hat\al$ is isomonodromic.  The term \ll monodromy data\rr refers to this formulation. The data consists of a collection of Stokes matrices, formal monodromy matrices, and connection matrices in the sense of o.d.e.\ theory, as described in \cite{GIL3},\cite{GILX}.

These two aspects of the tt* equations have been formalized in the concept of variation of TERP structure, in \cite{He03}.
The Toda aspect of the tt*-Toda equations gives rise to a third, \ll Lie theoretic\rrr, point of view. Namely, from the various symmetries of the Toda equations, it is known (from our previous work  \cite{GH1},\cite{GH2})
that the monodromy data can be expressed entirely in terms of two elements of the Lie group
$\SL_{n+1}\C$, for simplicity denoted in this introduction by $M,E$ (later on by $\tMz,\tE$).

The matrix $M$ represents all the data of the Stokes matrices, and
the matrix $E$ contains the additional data of the connection matrices.  They are constrained by the aforementioned symmetries, and have two important properties: $M$ is a regular element,
and $ME=EM$.
This means that the monodromy data $(E,M)$ is a subspace of the (group theoretic) \ll universal centralizer\rr of
$\SL_{n+1}\C$. Exploiting this fundamental fact is the basis of our approach.

Before stating our results, it is necessary to say a few words about properties of the solutions of (\ref{ost}).
(For more details we refer to \cite{GH17} and
\cite{GIL1}-\cite{GuLi14}.)
A (radial) solution defined on any nonempty open interval $a<\vert t\vert < b$ gives rise to a pair $(E,M)$. We refer to such solutions rather informally as \ll local solutions\rrr.  To use more precise terminology, we could remark that the equations
(\ref{ost})
have the Painlev\'e property, i.e.\ any local solution is the restriction of a \ll meromorphic\rr one.  Thus any local solution is, in a sense, globally defined for $0<\vert t\vert<\infty$, apart from at its poles --- of which there may be infinitely many. However, as we do not consider the location of such poles in this article, we shall not use this property.

We shall investigate the space of all local solutions, more concretely the space of all allowable pairs $(E,M)$, which we denote by
$\cS_{n+1}^{\text{local}}$.  For this purpose we consider the projection map to the second factor, which we denote by
\[
\pi: \cS_{n+1}^{\text{local}}\to \cM_{n+1}^{\text{local}},\quad (E,M) \mapsto M.
\]
Thus $\cM_{n+1}^{\text{local}}$ denotes the space of all allowable Stokes data.
We shall show that $\cM_{n+1}^{\text{local}}$ is a real affine space, and that
$\cS_{n+1}^{\text{local}}$ is a real (algebraic) manifold.

Although the map $\pi$ is not a fibre bundle, it exhibits well some of the analytic properties of solutions.
For example, it is known that for solutions which are smooth on regions of the form $0<\vert t\vert<\epsilon$ (\ll smooth near zero\rrr) the matrix $M$ is restricted to lie in a certain compact region, the subset where all eigenvalues of $M$ lie in $S^1$.
On the other hand, the solutions which are smooth on regions of the form $R<\vert t\vert<\infty$ (\ll smooth near infinity\rrr) all have $E=I$; they are a \ll cross section\rr of the map
$\pi: \cS_{n+1}^{\text{local}}\to \cM_{n+1}^{\text{local}}$.

The solutions of primary interest in \cite{CeVa91} are those which are globally smooth for $0<\vert t\vert<\infty$.  These global solutions have been studied in some detail in
\cite{GIL1}-\cite{GuLi14}, \cite{MoXX}-\cite{Mo14}. It is known that they are given precisely by the intersection of the previous two subspaces, i.e.\ they correspond to a \ll cross section\rr of $\pi$ over the compact region.
The space of such global solutions can, therefore,  be identified with that compact region,
which is in fact a convex polytope.
In the context of the
Hitchin equations these solutions constitute a certain space of harmonic metrics on the trivial bundle over $\C^\ast$ of rank $n+1$ with prescribed asymptotic behaviour at $t=0$ and $t=\infty$.
In the wider context of meromorphic connections and the Painlev\'e property, global solutions could be described as
meromorphic solutions with no poles on the positive real axis. We remark that the global solutions of the tt*-Toda equations for more general Lie algebras have been classified recently by Mochizuki (\cite{MoYY}).

Our first result (Theorem \ref{thm:universalsymplectic})
is that the universal centralizer is a holomorphic symplectic Lie groupoid (in particular this gives a new proof of the fact that it is a complex symplectic manifold).
Our main result (Corollary \ref{cor:result}) is that the space $\cS_{n+1}^{\text{local}}$ is
a real symplectic Lie subgroupoid of the universal centralizer; in particular it is a real
symplectic manifold. As a groupoid,
its space of units can be identified with
$\cM_{n+1}^{\text{local}}$, a real affine space.
As an essential step in the proof we shall show
that $\cS_{n+1}^{\text{local}}$ is the
set of common fixed points of two commuting groupoid involutions.

At the level of the Lie group $\SL_{n+1}\C$ these involutions are of the type used
to define symmetric spaces, and thus quite standard. However,
at the level of the monodromy data (see
Definition \ref{AdF}), they are essentially nonlinear, and we devote most of sections \ref{mon} and \ref{data} to deriving them in a suitable form.
We shall show that the
groupoid approach is ideally suited to dealing with them.

To place our results in context, we conclude by making brief comments on related results in the literature. We intend to discuss these in more detail in subsequent articles, along with generalizations and applications.

First and foremost, we remark that the definition of the
(group theoretic)
universal centralizer depends on a choice of Steinberg cross section of the set of regular conjugacy classes of the Lie group.  This has a Lie algebra theoretic, and perhaps more widely studied, cousin, the Kostant cross section of the set of regular adjoint orbits of the Lie algebra; it is not this, but the
group theoretic version that is needed to describe our monodromy data.
It was introduced by Lusztig, and is
an object of importance in representation theory (see \cite{Lu76}, \cite{Gi18}).
Although the Steinberg and Kostant cross sections are analogous, and have
analogous properties, they do not appear to be directly
related (cf.\ the remarks on this following Theorem 1.5 in \cite{St65}).

The complex symplectic structure of this group theoretic
universal centralizer
has nevertheless been studied by several authors. In \cite{BFM},
Bezrukavnikov, Finkelberg, and Mirkovic remarked that its symplectic structure is a result of
quasi-Hamiltonian reduction from the internal fusion double of \cite{AMM} --- we
comment further on this in Remark \ref{remarkX.X}.
The symplectic structure was also described
by Finkelberg and Tsymbaliuk in \cite{FiTs19}.
In \cite{Ba22}, Balibanu obtains
a symplectic structure
for the universal centralizer
from the point of view of Dirac structures in Poisson geometry.
Our Theorem \ref{thm:universalsymplectic}
gives a new proof, and provides the additional groupoid structure.

On the other hand, moduli spaces of meromorphic connections with irregular singularities over Riemann surfaces have been studied from the gauge theoretic point of view as \ll wild character varieties\rr by
Boalch
(see \cite{Bo01}, \cite{Bo07}),
building on fundamental work of Jimbo-Miwa-Ueno and Atiyah-Bott.  They, and
the corresponding spaces of monodromy data,  have a hyperK\"ahler structure (away from singular points), in particular a holomorphic symplectic structure.   There the monodromy data is described more conventionally in terms of Stokes matrices and connection matrices. Our data $M,E$ determine these matrices, but exist only in the tt*-Toda situation: $M$ is an $(n+1)$-th root of the monodromy, and $E$ is a normalized connection matrix, the normalization requiring knowledge of the global solutions of the tt*-Toda equations.

Concerning groupoids, we emphasize that our groupoid is not related to the \ll fundamental groupoid\rr whose representations were used in \cite{Bo01} to describe monodromy data.  It is closer to, but still quite distinct from, the symplectic double groupoid of \cite{Bo07}, Proposition 7.
We mention that our groupoid $\cS_{n+1}^{\text{local}}$ can be regarded as a subgroupoid of the symplectic groupoid discovered by Bondal (see \cite{Bo04}) in the context of
exceptional sequences of objects in triangulated categories
--- we explain this briefly in Remark \ref{remarkX.B}.

Acknowledgements:
The first author was partially supported by JSPS grants 18H03668 and 23H00083, and the second author was partially supported by MOST grants 109-2628-M-007-006-MY4 and 113-2115-M-007 -004 -MY3.
The authors thank (respectively) the
National Center for Theoretical Sciences and the Fields Institute
for support during their visits.
Both authors are very grateful to Eckhard Meinrenken for his advice regarding symplectic groupoids.
These results were first reported at
the workshop
\ll Hamiltonian Geometry and Quantization\rr
at the Fields Institute in July 2024
and the workshop
\ll Asymptotic Expansion of $\tau$-functions and Related Topics\rr
at the RIMS, Kyoto University, in February 2025.
The authors thank the organisers of these workshops for their hospitality.

\section{Monodromy data}\label{mon}

We begin by reviewing the definition and properties of the monodromy data.
Our primary reference for this is \cite{GH1}, although many of
the details can be found in \cite{GIL1}-\cite{GuLi14}.

\subsection{The connection form $\hat\al$.}\label{hatal}

We define below a connection form $\hat\al$, depending on a complex variable $\ze$,
whose matrix coefficients depend on functions
$w_0(\vert t\vert),\dots,w_n(\vert t\vert)$. We assume in this section that these functions satisfy equation (\ref{ost}) for
$\vert t\vert$ in
some nonempty open interval of the real line,
together with the periodicity, anti-symmetry, and radial conditions stated above.
Given such $w_0,\dots,w_n$, it is known that the monodromy data of the connection
does not depend on $\vert t\vert$ --- in this sense the
connection is isomonodromic.  This  monodromy data does depend on the particular solution $w_0,\dots,w_n$,
however, and it can be expected to contain important information about that solution.

It is also known that the map which assigns to a solution its
monodromy data is injective. The space of \ll all allowable monodromy data\rr is the
space of interest to us in this article --- it is the data which parametrizes solutions in the most intrinsic way.

Let us now review the definition of this data, beginning with the connection form $\hat\al$.
\begin{definition}\label{hatalpha-ze}
$\hat\al(\ze)=
\left[
-\tfrac{1}{\ze^2} W^T - \tfrac{1}{\ze} xw_x + x^2 W
\right]
d\ze$,
where
\[
w=
\bsp
 w_0  & &  \\
  &  \ddots& \\
  &  & w_n
\esp,
W=e^{-w} \Pi e^w,
\Pi=
{\tiny
\left(
\begin{array}{c|ccc}
 & \!1\! & &  \\
  &  &  \!\ddots\! & \\
  &  &   & \!1\!  \\
  \hline
 \!1\! &  &  &
\end{array}
\right)},
\ x=\vert t\vert,
\]
and $W^T$ denotes the transpose of $W$.
All matrices here are $n+1\times n+1$ matrices, and $\ze\in\C^\ast$.
\end{definition}
The equation for parallel sections of the connection $\nabla=d-\hat\al^T$ (the
dual connection to $d+\hat\al$) is
\begin{equation}\label{psizeta}
\Psi_\ze=
\left[
-\tfrac{1}{\ze^2} W - \tfrac{1}{\ze} xw_x + x^2 W^T
\right]
\Psi.
\end{equation}
This is a meromorphic o.d.e.\ which
has poles of order $2$ at $\ze=0$ and $\ze=\infty$.
We shall describe its monodromy data in the sense of classical o.d.e.\ theory (as
defined, for example, in \cite{FIKN06}).

\subsection{The Stokes matrices.}\label{stokes}

The Stokes matrices at $\ze=0$ are defined as in \cite{GH1} and \cite{GILX}, in the following way.
Equation (\ref{psizeta}) has a unique formal solution of the form
\[
\textstyle
\Psiz_f(\zeta) =  e^{-w}\, \Omega \left( I + \sum_{k\ge 1} \psiz_k \zeta^k \right) e^{\frac1\zeta  d_{n+1}},
\]
where
$
\Pi=(\de_{i,i+1})_{0\le i\le n}$ (as above) and
$\Omega=(\omega^{ij})_{0\le i,j\le n}$, $\om=e^{{2\pi \i}/{(n+1)}}$.
The formal monodromy is trivial --- see Appendix A of \cite{GH2}.
We refer to Appendix A of \cite{GIL2} for some properties of $\Pi$ and $\Om$,
the most important of which is that $\Om$ diagonalizes $\Pi$, i.e.\
$\Pi=\Om d_{n+1} \Om^{-1}$ where
$d_{n+1}=\diag(1,\om,\dots,\om^n)$.
We
have Stokes sectors
\[
\Omz_k
=\{ \zeta\in\tilde\C^\ast \st
\thz_k-\tfrac{\pi}{2}<\text{arg}\zeta <
\thz_{k-\scriptstyle\frac{1}{n+1}}+
\tfrac\pi2\}
\]
where
\[
\thz_k=
\begin{cases}
-\tfrac{1}{n+1}\pi - (k-1)\pi \quad\text{if $n$ is odd}
\\
-\tfrac{1}{2(n+1)}\pi - (k-1)\pi  \quad\text{if $n$ is even}
\end{cases}
\]
for $k\in\tfrac1{n+1}\Z$.
(The rays with angles $\thz_k$ are called singular directions.)
Each Stokes sector  $\Omz_k$ at $\ze=0$ in the $\ze$-plane supports a
holomorphic solution
$\Psiz_k$
which is uniquely characterized by the property
$\Psiz_k(\ze)\sim\Psiz_f(\ze)$
as $\zeta\to0$ in $\Omz_k$.

In terms of these choices, we define Stokes factors $\Qz_k$ by
\[
\Psiz_{k+\scriptstyle\frac1{n+1}} = \Psiz_k \Qz_k.
\]
The intersection
$
\Omz_k
\cap
\Omz_{k+\scriptstyle\frac1{n+1}}
$
of successive Stokes sectors
is the sector of width $\pi$ bisected by the singular direction $\thz_k$, so
one can regard
the Stokes factors as indexed by these $\thz_k$.
For $k\in\Z$ we define Stokes matrices
$\Sz_k$ by
\[
\Psiz_{k+1}=\Psiz_k \Sz_k.
\]
Thus
$
\Sz_k=\Qz_{k}\Qz_{k+\scriptstyle\frac1{n+1}}\Qz_{k+\scriptstyle\frac2{n+1}}\dots
\Qz_{k+\scriptstyle\frac{n}{n+1}}.
$
The monodromy of the solution $\Psiz_k$ is given by
$
\Psiz_k(e^{2\pi\i}\ze)=\Psiz_k(\ze) \Sz_k\Sz_{k+1}
$
(see section 4 of \cite{GILX}).

Similarly, at $\ze=\infty$, starting from the
formal solution
\[
\Psii_f(\zeta) =
\textstyle
e^w \Om^{-1} \left( I + \sum_{k\ge 1} \psii_k \zeta^{-k} \right) e^{x^2 \zeta  d_{n+1}},
\]
we obtain $\Psii_k$ and $\Qi_k$, $\Si_k$.

The Stokes factors satisfy various algebraic identities, which arise from the
various \ll symmetries\rr of the connection form $\hat\al$.
To express these, we shall use the following Lie algebra automorphisms of $\sl_{n+1}\C$
\[
\tau(X)=d_{n+1}^{-1} X d_{n+1},\ \
\si(X)=-\De\, X^T\De,\ \
c(X)=\De \bar X  \De,\ \
\th(X)=\bar X
\]
where
$\De=(\de_{i,n-i})$ is the anti-diagonal matrix with $1$ in
positions $(i,n-i)$, $0\le i\le n$, and $0$ elsewhere.
Thus $\si,c,\th$ are involutions, while $\tau$ has order $n+1$.

With this notation, the symmetries of the connection form $\hat\al$ are:

\noindent{\em Cyclic symmetry: }  $\tau(\hat\alpha(\zeta))=\hat\alpha(\om \zeta)$

\noindent{\em Anti-symmetry: }  $\sigma(\hat\alpha(\zeta))=\hat\alpha(-\ze)$

\noindent{\em $c$-reality: }  $c(\hat\alpha(\zeta))=\hat\alpha(1/(x^2\bar\ze\vphantom{\tfrac{a}{a}}))$

\noindent {\em $\theta$-reality: } $\th(\hat{\alpha}(\ze))=\hat{\alpha}(\bar{\ze})$.

\no These are easily verified from Definition \ref{hatalpha-ze}.

The $c$-reality condition allows us to express the matrices $\Qi_k$ in terms of the matrices $\Qz_k$ (section 4 of \cite{GILX}, Lemma 2.4 of \cite{GIL2}),
so we shall ignore $\Qi_k$.  The cyclic symmetry
leads to the following important formula:
\begin{proposition}\label{tauM}
(cf.\ section 2 of \cite{GIL2})
$\Qz_{k+\scriptstyle\frac2{n+1}} = \Pi\  \Qz_k \  \Pi^{-1}$.
\qed
\end{proposition}
Hence all Stokes factors are determined by any two consecutive
Stokes factors. It follows that
the monodromy matrix of $\Psiz_1$ can be rewritten as
\[
\Sz_1 \Sz_2 = (\Qz_1\Qz_{1+\scriptstyle\frac1{n+1}}\Pi)^{n+1}.
\]
In view of this, we introduce the following
notation:

\begin{definition}\label{matrixM}
$\Mz=\Qz_1\Qz_{1+\scriptstyle\frac1{n+1}}\Pi$.
\end{definition}

It is known (see section \ref{steinberg}) that the matrix $\Mz$ determines the individual matrices
$\Qz_1$ and $\Qz_{1+\scriptstyle\frac1{n+1}}$ (and hence, by the cyclic
symmetry formula, all $\Qz_k$). Thus the single matrix
$\Mz$ represents all the Stokes data at $\ze=0$.

A special feature of the
$A_n$-type tt*-Toda equations is that it is possible to
choose formal solutions so that all Stokes factors become
real matrices.
With this in mind
we define modified formal solutions at $\ze=0$ by
\[
\tPsiz_f(\zeta) =
\begin{cases}
\Psiz_f(\zeta)  \, d_{n+1}^{\frac12} \quad\text{if $n$ is odd}
\\
\Psiz_f(\ze) d_{n+1}^{-\frac12n} \quad\text{if $n$ is even}
\end{cases}
\]
where
$d_{n+1}^{\frac12}=\diag(1,\omega^{\frac12},\omega^{1},\dots,\omega^{\frac{n}2})$.
We obtain $\tPsiz_k$ and
$\tQz_k,\tSz_k$ as in the case of $\Psiz_k$.  (The reality property will be explained
later --- for the moment all Stokes factors are regarded as complex matrices).

From this definition we obtain
\begin{equation}\label{tQz}
\tQz_k=
\begin{cases}
d_{n+1}^{-\frac12} \Qz_k d_{n+1}^{\frac12} \quad\text{if $n$ is odd}
\\
 (d_{n+1})^{\frac n2} \Qz_k (d_{n+1})^{-\frac n2} \quad\text{if $n$ is even}
\end{cases}
\end{equation}
and then Proposition \ref{tauM} gives:

\begin{proposition}\label{tauMtilde} The cyclic symmetry property of $\tQz_k$ is
\[
\tQz_{k+\scriptstyle\frac2{n+1}} =
\begin{cases}
\hat\Pi \tQz_k \hat\Pi^{-1}
\quad\text{if $n$ is odd}
\\
\Pi \tQz_{k} \Pi^{-1}\quad\text{if $n$ is even}
\end{cases}
\]
where
$\hat\Pi
=\diag(1,\dots,1,-1)\Pi$.
\qed
\end{proposition}
As a tilde-analogue of $\Mz$ we introduce
\begin{equation}\label{tMz}
\tMz=
\begin{cases}
\tQz_1\tQz_{1+\scriptstyle\frac1{n+1}}\hat\Pi
\quad\text{if $n$ is odd}
\\
\tQz_1\tQz_{1+\scriptstyle\frac1{n+1}}
\Pi
\quad\text{if $n$ is even.}
\end{cases}
\end{equation}
It follows that
\begin{equation}\label{tMz^n+1}
(\tMz)^{n+1}=
\begin{cases}
 -\tSz_1 \tSz_2 \quad\text{if $n$ is odd}
\\
\ \ \tSz_1 \tSz_2 \quad\text{if $n$ is even.}
\end{cases}
\end{equation}
Let us write the characteristic polynomial of $\tMz$ as
\[
\sum_{i=0}^{n+1} s_i \mu^{i} \quad\text{if $n$ is odd},
\quad
\sum_{i=0}^{n+1} (-1)^is_i \mu^{i} \quad\text{if $n$ is even.}
\]
We refer to $s_1,\dots,s_n$ as the \ll Stokes parameters\rrr. Their role will be made
more explicit in the next section.

\subsection{The Steinberg cross section.}\label{steinberg}

We have obtained the Stokes parameters $s_1,\dots,s_n$ from $\tMz$.
Conversely, it is possible to recover $\tMz$ from  $s_1,\dots,s_n$.
The proof of this fact exploits the concept of \ll Steinberg cross section\rrr.
To explain this we need to quote further results from
\cite{GH1}.

First, it was shown in \cite{GH1} that $\tQz_k$ belongs to the unipotent subgroup of $\SL_{n+1}\C$ corresponding to a certain subset $\cRz_k\sub \De$ of the set of roots
$\De=\{ x_i-x_j\st 0\le i\ne j \le n \}$ (with respect to the standard Cartan subalgebra
consisting of matrices $\diag(x_0,\dots,x_n)$).  The roots in the subsets
$\cRz_1$, $\cRz_{1+\scriptstyle\frac1{n+1}}$
are listed in Proposition 3.4 and Proposition 3.9 of \cite{GH1} (we remark that
$\cRz_{k}$ was denoted by
$\cR  (\th_{k})$
in  \cite{GH1}).

The expression for $\tMz$ in formula (\ref{tMz}) motivates the next definition (Definitions 3.6 and 3.11 of \cite{GH1}). It is the space of all (complex) matrices having the shape of $\tMz$.

\begin{definition}\label{scriptM}
We define a subspace $\cM_{n+1}$ of $\SL_{n+1}\C$ by
\[
\cM_{n+1}=
\begin{cases}
\Bigl(
\Pi_{\al\in\cRz_1} U_\al
\Bigr)
\Bigl(
\Pi_{\al\in\cRz_{1+\scriptstyle\frac 1{n+1}}} U_\al
\Bigr)
\hat\Pi \quad\text{if $n$ is odd}
\\
\Bigl(
\Pi_{\al\in\cRz_1} U_\al
\Bigr)
\Bigl(
\Pi_{\al\in\cRz_{1+\scriptstyle\frac 1{n+1}}} U_\al
\Bigr)
\Pi \quad\text{if $n$ is even}
\end{cases}
\]
where $U_{x_i-x_j}= \exp(\C E_{i,j})$,
and $E_{i,j}$ is the matrix with $1$ in position $(i,j)$ and $0$ elsewhere.
\end{definition}

The following result was proved in section 4 of \cite{GH1},
based on work of Steinberg (\cite{St65}, \cite{St74}).

\begin{proposition}\label{stein}
$\cM_{n+1}$ is a Steinberg
cross section of the set of regular conjugacy classes of $\SL_{n+1}\C$.
The map which assigns to an element of $\cM_{n+1}$ the coefficients of
its characteristic polynomial
gives an isomorphism $\cM_{n+1}\cong \C^n$.
\qed
\end{proposition}

That is, each element of $\cM_{n+1}$ is regular (its minimal polynomial and characteristic polynomial coincide), and each conjugacy class in $\SL_{n+1}\C$ contains precisely one element of $\cM_{n+1}$.
More
precisely, it is the  Steinberg
cross section corresponding to the choice of simple roots
$\cRz_1\cup \cRz_{1+\scriptstyle\frac n{n+1}}$.  The  set of positive roots corresponding to these simple roots
is $\cRz_1\cup \dots\cup \cRz_{1+\scriptstyle\frac n{n+1}}$.

\begin{remark}\label{coxeter}
The set $\cRz_1\cup \cRz_{1+\scriptstyle\frac 1{n+1}}$ which
appears in Definition \ref{scriptM} may then be described
(see Corollary 5.6 of \cite{GH2})  as the subset of
the positive roots
$\cRz_1\cup \dots\cup \cRz_{1+\scriptstyle\frac n{n+1}}$
consisting of those
which become negative under the action of the
Coxeter element, the latter being represented by $\hat\Pi$ (if $n$ is odd) or $\Pi$  (if $n$ is even).
This action is given on roots by $x_i-x_j\mapsto x_{\de\cdot i}-x_{\de\cdot j}$
where $\de$ is the cyclic permutation $(n \, n-1 \dots 1\, 0)$.
All these facts were established (for any complex simple Lie group) in
\cite{GH2}, where it was shown that the Coxeter Plane provides a
convenient way of visualizing them.
\qed
\end{remark}

By construction, we have $\tMz\in\cM_{n+1}$. By Proposition \ref{stein} any
element of $\cM_{n+1}$ is of the form $\tMz$, for some $s_1,\dots,s_n\in\C$.
In particular the map
\[
\cM_{n+1}\to \C^n, \quad \tMz\mapsto (s_1,\dots,s_n)
\]
is a
bijection, so $\cM_{n+1}$ is an affine space isomorphic to $\C^n$.
It follows from the root theoretic description of $\tQz_k$
(and Definition \ref{matrixM})  that
$\tMz$ determines the individual matrices
$\Qz_1$ and $\Qz_{1+\scriptstyle\frac1{n+1}}$,
as asserted earlier.

In fact (section 6 of \cite{GILX})
there is an
explicit formula for the $\tQz_k$(and hence for $\tMz$) in terms
of $s_1,\dots,s_n$:

\begin{proposition}\label{Qfroms}
Let $E_{i,j}$ be the matrix with $1$ in position $(i,j)$ and $0$
elsewhere, where $0\le i,j\le n$.
Then
$
\tQz_k=
I + \sum_{x_i-x_j\in \cRz_k}  \, s_{i,j} E_{i,j}$,
where
\[
s_{i,j}=
\begin{cases}
s_{j-i}  \quad\text{if $i<j$}
\\
-s_{j-i}  \quad\text{if $i>j$}
\end{cases}
\quad\text{if $n$ is odd,}
\]
and
\[
s_{i,j}=
\begin{cases}
(-1)^{j-i} s_{j-i} \ \text{if $i<j$}
\\
-(-1)^{j-i} s_{j-i} \ \text{if $i>j$}
\end{cases}
\quad\text{if $n$ is even.}
\]
Here the $j-i$ in $s_{j-i}$ is interpreted mod $n+1$.
\qed
\end{proposition}

As this Stokes data arises from a local solution of (\ref{ost}), the connection form
$\hat\al$ has the symmetries listed in section \ref{hatal}. We have already described the effect of $\tau$ in Proposition \ref{tauM}. The effect of $\si$ and $\th$ is as follows:

\begin{proposition}\label{sithM} (cf.\ Lemma 2.3 of \cite{GIL2})

\noindent (i) Anti-symmetry:     $\tQz_{k+1} = (\tQz_k){}^{-T}$

\noindent (ii) $\theta$-reality:   $\tQz_k=
\begin{cases}
\tC
\left(
\bartQz_{ \scriptstyle\frac{2n}{n+1}-k}
\right)^{-1}
\tC
\quad\text{if $n$ is odd}
\\
C
\left(
\bartQz_{ \scriptstyle\frac{2n+1}{n+1}-k}
\right)^{-1}
C
\quad\text{if $n$ is even}
\end{cases}
$
\newline
where
$\tC
= \diag(1,-1,\dots,-1)C$ and
$
C=
\tiny
\left(
\begin{array}{c|ccc}
1 & & &  \\
\hline
  &  &  & 1 \\
  &  &  \iddots &  \\
  & 1 &  &
\end{array}
\right).
$
\qed
\end{proposition}

It follows from section 6 of \cite{GILX}
--- we shall give a direct proof later, in Theorem \ref{preservesM} ---
that these two conditions correspond, respectively, to the conditions
(a) $s_i=s_{n-i+1}$, (b) $s_1,\dots,s_n\in\R$.  Hence the space
obtained by imposing the symmetry conditions given by $\si$ and $\th$
can be described as follows:

\begin{definition}\label{scriptMlocal}
$\cM_{n+1}^{\text{\em local}}=
\{ \tMz \in \cM_{n+1} \st s_i=s_{n-i+1} \text{ and } s_1,\dots,s_n\in\R\}$.
\end{definition}

Thus $\cM_{n+1}^{\text{local}}$
may be identified with
$\R^{ \tfrac12(n+1)}$ if $n$ is odd, and
$\R^{ \tfrac12n}$ if $n$ is even.

\begin{remark}\label{localzero}
As the notation suggests (see \cite{GH1}), $\cM_{n+1}^{\text{local}}$ consists of the values of $\tMz$
which actually arise from local solutions of (\ref{ost}).
More precisely,  the natural map
from the space of local solutions to  $\cM_{n+1}^{\text{local}}$ is surjective (the fibres
consist of connection matrices, which we
define in the next section).  To see this (\cite{GH1}, \cite{GILX})
it suffices to consider solutions which are
smooth near $t=0$, i.e.\ on intervals of the form $(0,\eps)$.
\qed
\end{remark}

\subsection{The connection matrices.}\label{conn}

We define connection matrices  $E_k$ by
\[
\Psii_k=\Psiz_k E_k.
\]
It follows that
\begin{equation*}
E_k = \left(\Qz_{k-\scriptstyle\frac1{n+1}}\right)^{-1}  E_{k-\scriptstyle\frac1{n+1}}
\ \Qi_{k-\scriptstyle\frac1{n+1}}.
\end{equation*}
Thus the connection matrix $E_1$ determines
all other $E_k$.

It is shown in \cite{GIL3}, \cite{GILX}, that, when the monodromy data
arises from a solution $w$ which is
globally smooth on $\C^\ast$, the
corresponding value of $E_1$ must be
\[
E_1^{\id}=
\begin{cases}
\tfrac1{n+1} C \Qi_{\frac n{n+1}} \quad\text{if $n$ is odd}
\\
\tfrac1{n+1} C  \quad\text{if $n$ is even}
\end{cases}
\]
where $C$ was defined in Proposition \ref{sithM}.
Thus this special value of $E_1$
plays an important role, and we shall use it to \ll normalize\rr  $E_1$ in
the following way:
\begin{definition}\label{E}
$E=E_1 (E_1^{\id})^{-1}$.
\end{definition}
We can now say that all of the connection matrix data (beyond the Stokes data) is represented by the single matrix $E$.

Similarly, we define $\tE_k$ by
$\tPsii_k=\tPsiz_k\tE_k$.
We have
\[
\tE_k=
\begin{cases}
d_{n+1}^{-\frac12} E_k d_{n+1}^{-\frac12} \quad\text{if $n$ is odd}
\\
 (d_{n+1})^{\frac n2} E_k (d_{n+1})^{\frac n2} \quad\text{if $n$ is even.}
\end{cases}
\]
This gives
\[
\tE_1^{\id}=
\begin{cases}
\tfrac1{n+1} \tC \tQi_{\frac n{n+1}} \quad\text{if $n$ is odd}
\\
\tfrac1{n+1} C  \quad\text{if $n$ is even}
\end{cases}
\]
where $\tC$ is as in Proposition \ref{sithM}.
We normalize $\tE_1$ in the same way:
\begin{definition}\label{Etilde}
$\tE=\tE_1 (\tE_1^{\id})^{-1}$.
\end{definition}

To summarize: starting with a solution of (\ref{ost}) on some
nonempty open interval of the real line,
we have defined two matrices $\Mz,E$ (or $\tMz,\tE$) which represent the
monodromy data of the corresponding connection form $\hat\al$.
We have seen that the symmetries of $\tMz$ amount to
saying that $\tMz\in \cM_{n+1}^{\text{local}}$.
We shall now consider the analogous
symmetries of $\tE$.

\begin{proposition}\label{Ecyclic} The cyclic symmetry property of $\tE$ is $\tE\tMz=\tMz \tE$
(which is equivalent to $E\Mz=\Mz E$).
\end{proposition}

\begin{proof} The cyclic symmetry, applied to the definition of $E_k$, produces the relation
$d_{n+1}^{-1} E_k= \om \ (\Qi_k \Qi_{k+\frac1{n+1}} \Pi)
\
d_{n+1}^{-1} E_k\  (\Qi_k \Qi_{k+\frac1{n+1}} \Pi)$
(cf.\ Lemma 2.5 of \cite{GIL2}).  This is equivalent to
\[
E_k = (\Qz_k \Qz_{k+\frac1{n+1}} \Pi) E_k (\Qi_k \Qi_{k+\frac1{n+1}} \Pi).
\]
We obtain
(a) $E_1 = \Qz_1\Qz_{1+\scriptstyle\frac1{n+1}}\Pi \  E_1\
\Qi_1\Qi_{1+\scriptstyle\frac1{n+1}}\Pi$
and
(b) $E_1^{\id} = \Qz_1\Qz_{1+\scriptstyle\frac1{n+1}}\Pi \  E_1^{\id}\
\Qi_1\Qi_{1+\scriptstyle\frac1{n+1}}\Pi$.
In fact, (b) follows directly from the definition of $E_1^{\id}$ above.
Multiplication of (a) by the inverse of (b) then gives
\[
E_1 (E_1^{\id})^{-1}=
 \Qz_1\Qz_{1+\scriptstyle\frac1{n+1}}\Pi \
 E_1 (E_1^{\id})^{-1}
(\Qz_1\Qz_{1+\scriptstyle\frac1{n+1}}\Pi)^{-1},
\]
i.e.\ $E \Mz=\Mz E$. The tilde version $\tE\tMz=\tMz \tE$ follows immediately.
\end{proof}

The remaining symmetries are as follows. We
just present the results for the tilde versions:

\begin{proposition}\label{fibresymmetries}
The matrix $\tilde E$ has the following symmetries:

\noindent (i) Anti-symmetry:
$
\tilde E=
\begin{cases}
\hat\Pi^{(n+1)/2}  \tilde E^{-T}  \hat\Pi^{-(n+1)/2} \quad\text{if $n$ is odd}
\\
(\tSz_1)^T  \tilde E^{-T}  (\tSz_1)^{-T}  \quad\text{if $n$ is even}
\end{cases}
$

\noindent (ii) $\theta$-reality:
$
\tilde E=
\begin{cases}

\left(
\overline{\tilde C\tQz_{ {\scriptstyle\frac{n}{n+1}}  } }
\right)
\bar {\tilde E}
\left(
\overline{\tilde C\tQz_{ {\scriptstyle\frac{n}{n+1}}  } }
\right)^{-1}
\quad\text{if $n$ is odd}
\\
\ C
\bar {\tilde E}
C^{-1}
\quad\text{if $n$ is even}
\end{cases}
$

\noindent (iii) $c$-reality:  $\tilde E = \left( \bar{\tilde E}\right)^{-1}$.
\newline
\end{proposition}

\begin{proof}
All three can be obtained by applying
the symmetries of the functions $\Psiz_k,\Psii_k$ to the definition of $E_k$
(cf.\ section 2 of \cite{GIL2} for the case $n=3$).
As a typical example we sketch the proof of $\tilde E = (\bar{\tilde E})^{-1}$ in the case where $n$ is odd.
First, the $c$-reality condition gives
\[
\De \overline{\Psiz_f(1/(x^2\bar\ze)})(n+1)^{-1}d_{n+1}^{-1}C  = \Psii_f(\ze),
\]
because both sides are formal solutions of (\ref{psizeta}) of the same type,
hence must be equal.
It follows that
\begin{equation}\label{cforPsi}
\De \overline{\Psiz_{\scriptstyle\frac{2n+1}{n+1}-k}(1/(x^2\bar\ze)})(n+1)^{-1}d_{n+1}^{-1}C  = \Psii_k(\ze)
\end{equation}
for $\ze\in\Omi_k$ (i.e.\ for $1/(x^2\bar\ze)\in\Omz_{\scriptstyle\frac{2n+1}{n+1}-k}$),
because both sides are solutions of (\ref{psizeta}) which are asymptotic as $\ze\to\infty$ to
the same formal solution,
hence must be equal.  Next, we apply this to $\Psii_l=\Psiz_l E_l$ with
$l=k, l=\scriptstyle\frac{2n+1}{n+1}-k$.
Comparing the resulting formula with (\ref{cforPsi}), we conclude that
$(n+1)^2 d_{n+1}^{-1} E_k C d_{n+1} \overline{ E_{\scriptstyle\frac{2n+1}{n+1}-k} } C = I$.
Finally, putting $k=1$, and using the definition $E=E_1 (E_1^{\id})^{-1}$,
we obtain
$(d_{n+1}^{-1}E) \overline{(d_{n+1}^{-1}E)}=I$. The tilde version of this is
the required formula
$\tilde E \bar{\tilde E} =I$.
\end{proof}

As we have seen, the symmetries of $\tMz$ imply that it is a real matrix.
Although $\tE$ is
not in general real, (ii) shows that its characteristic polynomial is real.

In the next section we shall consider the space of all {\em pairs} $(\tE,\tMz)$. In
order to put the symmetries of $\tE$ and $\tMz$ on the same footing we shall express
the symmetries of $\tMz$ in a similar way to those of $\tE$:

\begin{proposition}\label{basesymmetries}
The matrix $\tMz$ has the following symmetries:

\noindent (i) Anti-symmetry:
$
\tMz=
\begin{cases}
\hat\Pi^{(n+1)/2}  (\tMz)^{-T}  \hat\Pi^{-(n+1)/2} \quad\text{if $n$ is odd}
\\
(\tSz_1)^T  (\tMz)^{-T}  (\tSz_1)^{-T}  \quad\text{if $n$ is even}
\end{cases}
$

\noindent (ii) $\theta$-reality:
$
\tMz=
\begin{cases}
\left(
\overline{\tilde C\tQz_{ {\scriptstyle\frac{n}{n+1}}  } }
\right)
(\bar {\tilde M}^{(0)})^{-1}
\left(
\overline{\tilde C\tQz_{ {\scriptstyle\frac{n}{n+1}}  } }
\right)^{-1}
\quad\text{if $n$ is odd}
\\
\vphantom{  \dfrac12  }
\ C
(\bar {\tilde M}^{(0)})^{-1}
C^{-1}
\quad\text{if $n$ is even.}
\end{cases}
$
\end{proposition}

\begin{proof}
As in the case of $\tE$ these properties follow from
the definition of $\tQz_k,\tMz_k$
and the symmetries of the functions $\Psiz_k,\Psii_k$.
The next theorem will give an independent proof,
together with a reformulation that will be useful in the next section.
\end{proof}

\begin{definition}\label{AdF}
Given $X\in\SL_{n+1}\C$ and a map
\[
F:\C^n\to \SL_{n+1}\C, \quad s=(s_1,\dots,s_n)\mapsto F(s)=F(s_1,\dots,s_n),
\]
we write $\Ad_{F(s)}X = F(s)XF(s)^{-1}$.
We shall take (in the next theorem)
$F=F_\si$ or $F=F_\th$ where
\[
F_\si(s)=
\begin{cases}
\hat\Pi^{(n+1)/2}  \quad\text{if $n$ is odd}
\\
(\tSz_1)^T   \quad\text{if $n$ is even}
\end{cases}
\quad
F_\th(s)=
\begin{cases}
\overline{\tilde C\tQz_{ {\scriptstyle\frac{n}{n+1}}  } }  \quad\text{if $n$ is odd}
\\
C   \quad\text{if $n$ is even.}
\end{cases}
\]
Note that $F_\si$ and $F_\th$ are indeed functions of $s$ (constant in two of the above four cases).
By Proposition \ref{stein},  points $s\in\C^n$ correspond bijectively to points $A\in \cM_{n+1}$, so we may, alternatively,  regard $F$ as a map
$F:\cM_{n+1}\to \SL_{n+1}\C$.
\end{definition}

\begin{theorem}\label{preservesM}  Consider the Stokes data $\tQz_k,\tMz$
defined in this section, assuming the cyclic symmetry condition
in Proposition \ref{tauMtilde}, but without assuming the anti-symmetry or $\theta$-reality
conditions in Proposition \ref{sithM}.

\no(i) For any $\tMz\in\cM_{n+1}$, we have $\Ad_{F_\si(s)} (\tMz)^{-T} \in \cM_{n+1}$.
If $\tMz$ corresponds to $(s_1,\dots,s_n)\in\C^n$, then $\Ad_{F_\si(s)} (\tMz)^{-T}$
corresponds to $(t_1,\dots,t_n)\in\C^n$,
where $t_i=s_{n-i+1}$.
Furthermore, $\Ad_{F_\si(s)} (\tMz)^{-T}=\tMz$ if and only if
the anti-symmetry condition of Proposition \ref{sithM} holds.

\no(ii)  For any $\tMz\in\cM_{n+1}$, we have
$\Ad_{F_\th(s)} (\overline{\tMz  })^{-1} \in \cM_{n+1}$.
If $\tMz$ corresponds to $(s_1,\dots,s_n)\in\C^n$, then $\Ad_{F_\th(s)} (\overline{\tMz  })^{-1}$
corresponds to $(t_1,\dots,t_n)\in\C^n$,
where $t_i=\bar s_{n-i+1}$.
Furthermore, $\Ad_{F_\th(s)} (\overline{\tMz  })^{-1}=\tMz$ if and only if
the $\th$-reality condition of Proposition \ref{sithM} holds.
\end{theorem}

\begin{proof} (i) Let $n+1=2l$, $l\in\N$.  As $\tMz=
\tQz_1\tQz_{1+\scriptstyle\frac1{n+1}}\hat\Pi$, we have
\[
\Ad_{F_\si} (\tMz)^{-T}=
\hat\Pi^l (\tQz_1)^{-T} (\tQz_{1+\scriptstyle\frac1{n+1}})^{-T} \hat\Pi\ \hat\Pi^{-l}
\]
(note that $\hat\Pi^{-T}=\hat\Pi$).  The cyclic symmetry for $\tQz_k$
is
$\hat\Pi (\tQz_k)^{-T} \hat\Pi^{-1} = (\tQz_{k+\scriptstyle\frac1{l}})^{-T}$, so we obtain
\[
\Ad_{F_\si} (\tMz)^{-T}=
(\tQz_2)^{-T} (\tQz_{2+\scriptstyle\frac1{n+1}})^{-T} \hat\Pi.
\]
This belongs to $\cM_{n+1}$ because \ll transpose\rr sends positive roots to their negatives, i.e.\
it sends $\cR_k$ to $\cR_{k+1}$ (=$\cR_{k-1}$), as does the map
$\tQz_k\mapsto \tQz_{k+1}$.  Thus,
$(\tQz_2)^{-T}$ has the same root spaces as $\tQz_1$, and
$(\tQz_{2+\scriptstyle\frac1{n+1}})^{-T}$ has the same root spaces as
$\tQz_{1+\scriptstyle\frac1{n+1}}$.
(Note that \ll inverse\rr reverses the signs of root vectors, but does not change the root spaces.)

To determine the effect on $(s_1,\dots,s_n)$ it suffices (by Proposition \ref{stein}) to consider the
effect on the eigenvalues.  Evidently the eigenvalues of $\Ad_{F_\si} (\tMz)^{-T}$ are the reciprocals
of the eigenvalues of $\tMz$, so the coefficients of the characteristic polynomial are reversed.
Thus $\Ad_{F_\si} (\tMz)^{-T}$ corresponds to $(s_n,\dots,s_1)$, as required.

Let $n+1=2l+1$, $l\in\N$.   This time
$\tMz=
\tQz_1\tQz_{1+\scriptstyle\frac1{n+1}}\Pi$
and we have
$\Ad_{F_\si} (\tMz)^{-T}=$
\[
\left(
\tQz_1 \cdots \tQz_{1+\scriptstyle\frac n{n+1}}
\right)^T
\
(\tQz_1)^{-T} (\tQz_{1+\scriptstyle\frac1{n+1}})^{-T} \Pi
\
(\tQz_1)^{-T} \cdots (\tQz_{1+\scriptstyle\frac n{n+1}})^{-T}.
\]
Using the cyclic symmetry
$\Pi (\tQz_k)^{-T} \Pi^{-1} = (\tQz_{k+\scriptstyle\frac1{l}})^{-T}$
this reduces to
\[
\Ad_{F_\si} (\tMz)^{-T}=
(\tQz_2)^{-T} (\tQz_{2+\scriptstyle\frac1{n+1}})^{-T} \Pi
\]
which is again in $\cM_{n+1}$. Again the effect on  $(s_1,\dots,s_n)$ is
to reverse their order.

The \ll furthermore\rr statement follows from this calculation.
Namely, the proof shows that
$\Ad_{F_\si} (\tMz)^{-T}=\tMz$ if and only if
$(\tQz_2)^{-T}=\tQz_1$ and
$(\tQz_{2+\scriptstyle\frac1{n+1}})^{-T} = \tQz_{1+\scriptstyle\frac1{n+1}}$.
By the cyclic symmetry, this is equivalent to having
$\tQz_{k+1}=(\tQz_k)^{-T}$.
This completes the proof of (i).

(ii) Let $n+1=2l$, $l\in\N$. Then
\begin{align*}
\Ad_{F_\th} (\overline{\tMz  })^{-1}
&=
\overline{\tilde C\tQz_{ {\scriptstyle\frac{n}{n+1}}  } }
(\overline{\tQz_1\tQz_{1+ {\scriptstyle\frac{1}{n+1}}  }\hat\Pi}  )^{-1}
(\overline{\tilde C\tQz_{ {\scriptstyle\frac{n}{n+1}}  } })^{-1}
\\
&=
\overline{\tilde C\tQz_{ {\scriptstyle\frac{n}{n+1}}  } }
\hat\Pi^{-1}
(\overline{\tQz_{1+ {\scriptstyle\frac{1}{n+1}} } })^{-1}
(\overline{\tQz_1 })^{-1}
\overline{\tilde C\tQz_{ {\scriptstyle\frac{n}{n+1}}  } }^{-1}
\\
&=
\overline{\tilde C\tQz_{ {\scriptstyle\frac{n}{n+1}}  } }
\
(\overline{\tQz_{{\scriptstyle\frac{n}{n+1}} } })^{-1}
(\overline{\tQz_{{\scriptstyle\frac{n-1}{n+1}} } })^{-1} \hat\Pi^{-1}
\
(\overline{\tilde C\tQz_{ {\scriptstyle\frac{n}{n+1}}  } })^{-1}
\end{align*}
(using the cyclic symmetry at the last step).   This gives
\begin{align*}
\Ad_{F_\th} (\overline{\tMz  })^{-1}
&=
\tC
\
(\overline{\tQz_{{\scriptstyle\frac{n-1}{n+1}} } })^{-1} \hat\Pi^{-1}
\
(\overline{\tQz_{ {\scriptstyle\frac{n}{n+1}}  } })^{-1} \tilde C
\\
&=
\tC
\
(\overline{\tQz_{{\scriptstyle\frac{n-1}{n+1}} } })^{-1}
\
(\overline{\tQz_{ {\scriptstyle\frac{n-2}{n+1}}  } })^{-1}\hat\Pi^{-1} \tilde C
\end{align*}
(using the cyclic symmetry again).  This may be written
\[
\Ad_{F_\th} (\overline{\tMz  })^{-1}=
\tC (\overline{\tQz_{{\scriptstyle\frac{n-1}{n+1}} } })^{-1}\tC
\
\tC(\overline{\tQz_{{\scriptstyle\frac{n-2}{n+1}} } })^{-1}\tC
\
\hat\Pi
\]
using the identity $\tC \hat\Pi^{-1}\tC=\hat\Pi$.  It will follow that this belongs to $\cM_{n+1}$ if we show
(a) that $\tC (\overline{\tQz_{{\scriptstyle\frac{n-1}{n+1}} } })^{-1}\tC$
has the same root spaces as $\tQz_1$, and
(b) that $\tC(\overline{\tQz_{{\scriptstyle\frac{n-2}{n+1}} } })^{-1}\tC$
has the same root spaces as
$\tQz_{1+\scriptstyle\frac1{n+1}}$.

To prove (a) and (b) we use the fact --- easily established by direct calculation --- that the effect of $X\mapsto \tC X \tC$
on root spaces is
\[
(i,j)\mapsto
\begin{cases}
(n+1-i,n+1-j)\quad i,j\ne 0
\\
(0,n+1-j)\quad i= 0
\\
(n+1-i,0)\quad j= 0.
\end{cases}
\]
Here we use the convention of \cite{GH1} that $(i,j)$ denotes the root, or root space, of
$x_i-x_j$. (Note that \ll inverse\rr and \ll conjugate\rr do not change the root spaces.). We also use the fact (see Remark \ref{coxeter} above, and Corollary 5.5 of \cite{GH2}) that the
cyclic permutation $\de=(n \, n-1 \dots 1\, 0)$ sends
$\cRz_{k}$ to $\de\cdot \cRz_{k}=\cRz_{k+\scriptstyle\frac 2{n+1}}$.
Let us write $\tC \, \cRz_{k}\tC$ for the result of applying $X\mapsto \tC X \tC$
to (the root spaces of) $\cRz_{k}$.

To prove (a), we must show that
$\tC \, \cRz_{{\scriptstyle\frac{n-1}{n+1}} } \tC$ coincides
with $\de\cdot \cRz_{{\scriptstyle\frac{n-1}{n+1}} } (=  \cRz_{1})$.
Equivalently, we must show that
$\de^{-1}\cdot \cRz_{1} = \tC \, \cRz_{1} \tC$.
This condition involves only the set $\cRz_{1}$, which is given explicitly in Proposition 3.4 of \cite{GH1}.  Using this, it is easy to verify that the required condition holds.

To prove (b), we must show that
$\tC \, \cRz_{{\scriptstyle\frac{n-2}{n+1}} } \tC$ coincides
with $\de^2\cdot \cRz_{{\scriptstyle\frac{n-2}{n+1}} } (=  \cRz_{1+{\scriptstyle\frac{1}{n+1}}})$.
Equivalently, we must show that
$\de^{-2}\cdot \cRz_{1+{\scriptstyle\frac{1}{n+1}}}
= \tC \, \cRz_{1} \tC$. Again, this can be verified using the description of $\cRz_{1+{\scriptstyle\frac{1}{n+1}}}$
in Proposition 3.4 of \cite{GH1}.

The eigenvalues of $\Ad_{F_\th} (\overline{\tMz  })^{-1}$ are the reciprocals of the complex conjugates
of the eigenvalues of $\tMz$, so the coefficients of the characteristic polynomial are reversed and conjugated. Thus $\Ad_{F_\th} (\overline{\tMz  })^{-1}$
corresponds to $(\bar s_n,\dots,\bar s_1)$, as required.

Let $n+1=2l+1$, $l\in\N$. Then
\begin{align*}
\Ad_{F_\th} (\overline{\tMz  })^{-1}
&=
C(\overline{\tQz_1\tQz_{1+ {\scriptstyle\frac{1}{n+1}}  }\Pi}  )^{-1} C
\\
&=
C \Pi^{-1}
(\overline{\tQz_{1+ {\scriptstyle\frac{1}{n+1}} } })^{-1}
(\overline{\tQz_1 })^{-1} C
\\
&=
C
(\overline{\tQz_{{\scriptstyle\frac{n}{n+1}} } })^{-1}
(\overline{\tQz_{{\scriptstyle\frac{n-1}{n+1}} } })^{-1} \Pi^{-1}C
\end{align*}
(using the cyclic symmetry at the last step).  This may be written
\[
\Ad_{F_\th} (\overline{\tMz  })^{-1}=
C (\overline{\tQz_{{\scriptstyle\frac{n}{n+1}} } })^{-1}C
\
C(\overline{\tQz_{{\scriptstyle\frac{n-1}{n+1}} } })^{-1}C
\
\Pi
\]
using the fact that $C \Pi^{-1}C=\Pi$.  This belongs to $\cM_{n+1}$ by a similar argument
to that used in (i), i.e.\ by establishing analogues of (a) and (b) above.
Using analogous notation, this amounts to establishing
(a) $C \, \cRz_{1} C = \de^{-1}\cdot \cRz_{1+ {\scriptstyle\frac{1}{n+1}} }$
and
(b) $C \, \cRz_{1+ {\scriptstyle\frac{1}{n+1}} } C=
\de^{-1}\cdot \cRz_{1}$.
These can be verified by using the explicit expressions for
$\cRz_{1}$, $\cRz_{1+ {\scriptstyle\frac{1}{n+1}} }$
in Proposition 3.9 of \cite{GH1}.
Again the effect on  $(s_1,\dots,s_n)$ is
to reverse and take complex conjugates.

As in the case of (i), the \ll furthermore\rr statement follows from this calculation.
\end{proof}

In particular this theorem shows that the anti-symmetry condition of
Proposition \ref{sithM} is {\em equivalent} to the anti-symmetry condition of
Proposition \ref{basesymmetries}, and similarly for the $\th$-reality conditions.

\section{The space of monodromy data.}\label{data}

In this section we introduce the space of \ll all allowable monodromy data\rrr, modelled on the properties
of $\tMz$ and $\tE$ given in the previous section, and establish its relation with
the universal centralizer.

\subsection{The space $\cS_{n+1}^{\text{local}}$.}\label{locsol}

\begin{definition}\label{thespace}
Let
$
\cS_{n+1}^{\text{\em local}}
$
be the set of
$(B,A)$ in $\SL_{n+1}\C \times
\cM_{n+1}$
such that
$BA=AB$, and
such that
$B$ satisfies conditions (i) and (ii) in Proposition \ref{fibresymmetries}
with $B=\tE$, and
$A$ satisfies conditions (i) and (ii) in Proposition \ref{basesymmetries}
with $A=\tMz$.
We call $\cS_{n+1}^{\text{\em local}}$
the {\em space of  monodromy data} (of the
tt*-Toda equations).
\end{definition}

The conditions on $\tE$ and $\tMz$
in Proposition \ref{fibresymmetries} and Proposition \ref{basesymmetries}
involve additional matrices
$\tQz_{ {\scriptstyle\frac{n}{n+1}}  }$ and
$\tSz_1 = \tQz_1 \cdots \tQz_{1+\scriptstyle\frac n{n+1}}$, but
these matrices depend on $s$, so we can interpret both of them
as functions of $s$, or, equally, of  $A=\tMz$. Thus
the above definition makes sense.

We have (other than in Remark \ref{localzero}) not yet given any explicit relation between
local solutions of (\ref{ost}) and their corresponding monodromy data.
To justify the notation $\cS_{n+1}^{\text{local}}$,
we make some comments on this, below. From now on, however, we shall
study the space $\cS_{n+1}^{\text{local}}$ in its own right, without
using any properties of the corresponding solutions.

\begin{remark}\label{remarkX.1}
(i) As the notation suggests, $\cS_{n+1}^{\text{local}}$ can
be regarded as the set of values of $(\tE,\tMz)$
which actually arise from local solutions of (\ref{ost}), or rather, from meromorphic solutions (possibly
interpreted as meromorphic sections of holomorphic bundles) together with a reality condition.
This is a consequence of the
well-posedness of the Riemann-Hilbert problem (as formulated in
\cite{GILX}, section 7).  To give a precise statement is
beyond the scope of this article, but we note that
a Riemann-Hilbert correspondence for general meromorphic connections
(without any reality conditions)
was given by Boalch \cite{Bo01} (in
the analytic context) and by Inaba and Saito \cite{IS13} (in
the algebraic context).

To write the equivalence in more explicit terms
and give a precise description of $\cS_{n+1}^{\text{local}}$
it would be necessary to relate all solutions
explicitly with their monodromy data. To our knowledge, this has been
carried out (for the tt*-Toda equations)
only in the case $n=1$, in \cite{GH17}. There, a stratification
of $\cS_{n+1}^{\text{local}}$ was given, each stratum corresponding
to a particular form of asymptotic behaviour of the solution at $t=0$ and $t=\infty$.
For $n=1$ this analysis shows that there is one stratum which does
{\em not} correspond to {\em any} local (real) solutions of (\ref{ost}),
namely the stratum given by $E_1=- E_1^{\id}$ and $s_1\in [-2,2]$.
Such monodromy data corresponds to solutions which take
values in $\frac12\pi\ii + \R$. We shall accept such solutions as \ll local solutions
of (\ref{ost})\rr (also when $n>1$).

(ii) The
solutions which are smooth near $t=0$ have been studied in detail, in
\cite{GIL1}-\cite{GuLi14} and \cite{GH1}-\cite{GH2}.  The \ll generic\rr solutions
of this kind form an open subset of $\cS_{n+1}^{\text{local}}$,  and for these
solutions all monodromy data (in particular, the matrices $\tMz$ and $\tE$) was computed
explicitly in terms of the asymptotics of the solutions at $t=0$.
The
solutions which are smooth near $t=\infty$ were treated similarly,
although these do not contain any open subset of $\cS_{n+1}^{\text{local}}$.
Although we are leaving open the problem of characterizing all points of $\cS_{n+1}^{\text{local}}$
in terms of solutions of (\ref{ost}) when
$n>1$, we would like to emphasize that $\cS_{n+1}^{\text{local}}$ is the natural space
to investigate geometrically.
\qed
\end{remark}

\begin{remark}\label{remarkX.2}
We have omitted $c$-reality (condition (iii) of Proposition \ref{fibresymmetries})
from the definition of $\cS_{n+1}^{\text{local}}$.
This is partly because the $c$-reality condition is not relevant to $\tMz$ (it was used
to reduce the Stokes data at $\ze=\infty$ to that at $\zeta=0$), and partly because the
$c$-reality condition for $\tE$ turns out to be automatically satisfied, at least in the case of solutions which are smooth near $t=0$.
This follows from the explicit formula for $E_1$ given in section 3 of \cite{GILX}.
Such solutions
form an open subset of $\cS_{n+1}^{\text{local}}$.
Furthermore, we shall prove in the next section that $\cS_{n+1}^{\text{local}}$ is a smooth
(real) algebraic manifold. If (as we expect, and have verified in low dimensions) it turns
out to be connected, then it would follow 
that the $c$-reality condition holds automatically on the whole of $\cS_{n+1}^{\text{local}}$.
\qed
\end{remark}

\subsection{The universal centralizer.}\label{uc}

\begin{definition}\label{ucdef} (\cite{Lu76}) Let $\Sigma$ be any Steinberg cross section of
$\SL_{n+1}\C$.
The universal centralizer with respect to $\Sigma$ is the algebraic variety
$
\cZ_{n+1}^\Sigma=
\{ (B,A) \st
B\in  \SL_{n+1}\C,
A\in\Sigma,
BA=AB
\}
$.
\end{definition}

Thus  $\cZ_{n+1}^\Sigma$ is the union of the centralizers of the elements
of $\Sigma$.
The standard example of  a Steinberg cross section of $\SL_{n+1}\C$ is
the space of \ll companion matrices\rrr.

\begin{remark}\label{remarkX.X} The authors of \cite{BFM} (2024 version)
remarked that the natural map from the universal centralizer to the quasi-Hamiltonian reduction of the internal fusion double is a birational isomorphism. 
This is the character variety $\Hom(\pi_1(\Sigma_1),G)/G$, where $\Sigma_1$ denotes here the genus $1$ torus. 
In fact, more can be said. The natural map identifies the universal centralizer with the non-Hausdorff space $\{ (A,B)\in G^\reg\times G ~|~AB=BA\}/G$. 
However, the image is not (fully) contained in the smooth part of the quasi-Hamiltonian reduction, which requires the representations to be reductive in order to have a Hausdorff quotient(cf.\ \cite{Go84}). This requires  $A$ and $B$ to be semisimple, but in the monodromy data of the tt* equations the non-semisimple points play an important role.
\qed
\end{remark}

By Proposition \ref{stein},
the space $\cM_{n+1}$ (in Definition \ref{scriptM})
is a Steinberg cross section, so
our space $\cS_{n+1}^{\text{local}}$ is a subspace of
$\cZ_{n+1}^{\cM_{n+1}}$. To simplify notation, from now on we shall just write
$\cZ_{n+1}$ for $\cZ_{n+1}^{\cM_{n+1}}$.
We shall describe the subspace precisely,  in terms of the following maps.

\begin{definition}\label{morphism} Let $\cZ_{n+1}$ be the
universal centralizer with respect to $\Sigma=\cM_{n+1}$.
We define maps $\bsi,\bth:\cZ_{n+1}\to \cZ_{n+1}$ by
\[
\bsi(B,A)= (\Ad_{F_\si(A)}  B^{-T}, \Ad_{F_\si(A)}  A^{-T})
\]
and
\[
\bth(B,A)= (\Ad_{F_\th(A)} \bar B, \Ad_{F_\th(A)} \bar A^{-1}),
\]
where
$F_\si$ and $F_\th$ are as in Definition \ref{AdF} (and we use the convention explained there
of writing $F(A)$ instead of $F(s)$).
Restricting to the second components, we define maps $\bsi_0,\bth_0:\cM_{n+1}\to \cM_{n+1}$ by
$\bsi_0(A)=\Ad_{F_\si(A)}  A^{-T}$,
$\bth_0(A)=\Ad_{F_\th(A)} \bar A^{-1}$.
\end{definition}

The maps $\bsi,\bth$ are well-defined because
(a)  if $B,A$ commute,
then so do  $B^{-T},A^{-T}$ and  $\bar B,\bar A^{-1}$, and
(b) by Theorem \ref{preservesM}
their second components $\bsi_0,\bth_0$ preserve $\cM_{n+1}$.

Let us denote the fixed point sets of $\bsi,\bth$ by
\[
\Fix \bsi = \cZ_{n+1}^{\bsi},\quad
\Fix \bth = \cZ_{n+1}^{\bth}.
\]
Then our space of monodromy data is
\[
\cS_{n+1}^{\text{local}}=\cZ_{n+1}^{\bsi} \cap \cZ_{n+1}^{\bth}
\quad (=\cZ_{n+1}^{\bsi,\bth}, \ \text{for brevity} ).
\]
This formulation will be used in the next section.

It should be noted that, although the original maps $\si,\th$
are induced by standard Lie group/Lie algebra involutions,
the maps $\bsi,\bth$ on monodromy data (in two of the four cases) are not. It is not immediate
--- in fact, it may at first seem unlikely --- that they are involutions, but in fact they are, and
we end this section by giving a proof.

\begin{theorem}\label{comminv}
The maps $\bsi,\bth$ are (commuting) involutions on $\cZ_{n+1}$.
\end{theorem}

\begin{proof}
Proof of $\bsi\circ\bsi=\id$:
When $n$ is odd we have $\FA=\phn$. As $(\phn)^2=-I$, it is obvious that
$\bsi\circ\bsi(B,A)=(B,A)$.
When $n$ is even we have $F_\si(A)=(\tSz_1)^T$.
By Theorem \ref{preservesM},
$F_\si(\bsi_0(A))$ is the result of reversing $s=(s_1,\dots,s_n)$ in $F_\si(A)$;
let us denote this by $(\tSzr_1)^T$.
Now, the proof of Theorem \ref{preservesM} shows that
$
\tQzr_{k } = (\tQz_{k+1})^{-T}
$
and hence
$
\tSzr_{1} = (\tSz_{2})^{-T}.
$
Using this, we obtain
\[
F_\si(\bsi_0(A)) F_\si(A)^{-T} = (\tSz_{2})^{-1} (\tSz_{1})^{-1} = (\tSz_{1} \tSz_{2})^{-1}.
\]
By equation (\ref{tMz^n+1}) we have
$\tSz_1 \tSz_2=A^{n+1}$. This commutes with both $A$ and $B$, so the result follows.

Proof of $\bth\circ\bth=\id$: When $n$ is even we have $F_\th(A)=C$. As $C^2=I$, it is obvious that
$\bth\circ\bth(B,A)=(B,A)$.
When $n$ is odd, we have $F_\th(A)=
\tC \overline{\tQz_{{\scriptstyle\frac{n}{n+1}} } }$. By Theorem \ref{preservesM},
$F_\th(\bsi_0(A))$ is $\tC \tQzr_{{\scriptstyle\frac{n}{n+1}} }$. This gives
\[
F_\th(\bsi_0(A)) \overline{F_\si(A)} = \tC \tQzr_{{\scriptstyle\frac{n}{n+1}} }
\tC \tQz_{{\scriptstyle\frac{n}{n+1}} }.
\]
The proof of Theorem \ref{preservesM} shows that
$\tC
\left(
\tQzr_{ \scriptstyle\frac{n}{n+1}}
\right)^{-1}
\tC
=
\tQz_{  \scriptstyle\frac{n}{n+1} }.
$
Thus $F_\th(\bsi_0(A)) \overline{F_\si(A)}=I$, and the result follows.

Finally we prove that $\bsi\circ\bth=\bth\circ\bsi$. Let us write
\[
\bth\circ\bsi(B,A)= (\Ad_P \, (\bar B)^{-T} \!, \Ad_P \, (\bar A)^{T} \!,
\bsi\circ\bth(B,A)= (\Ad_Q \, (\bar B)^{-T} \!, \Ad_Q\,  (\bar A)^{T})
\]
where
$
P=F_\th(\bsi_0(A)) \overline{F_\si(A)}
$,
$
Q= F_\si(\bth_0(A)) F_\th(A)^{-T}
$.

Proof of $\bsi\circ\bth=\bth\circ\bsi$ when $n$ is odd: We have $F_\si(A)=\phn$ and $F_\th(A)=
\tC \overline{\tQz_{{\scriptstyle\frac{n}{n+1}} } }$. By Theorem \ref{preservesM},
$F_\th(\bsi_0(A))$ is $\tC \overline{\tQzr_{{\scriptstyle\frac{n}{n+1}} } }$.
The proof of Theorem \ref{preservesM} shows that
$
\tQzr_{{\scriptstyle\frac{n}{n+1}} } = (\tQz_{{\scriptstyle 1+\frac{n}{n+1}} })^{-T}.
$
Using this, we obtain
\[
\bar P= \tC \, (\tQz_{{\scriptstyle 1+\frac{n}{n+1}} })^{-T} \ \phn,\quad
\bar Q= \phn \tC \, (\tQz_{{\scriptstyle \frac{n}{n+1}} })^{-T}.
\]
We claim that $P=-Q$, which will complete the proof.  The claim follows from the cyclic symmetry
$\tQz_{k+\scriptstyle\frac2{n+1}} = \hat\Pi \tQz_k \hat\Pi^{-1}$ and the
identities
$\tC \hat\Pi \tC=\hat\Pi^{-1}$, $\hat\Pi^2=-I$.

Proof of $\bsi\circ\bth=\bth\circ\bsi$ when $n$ is even:
We have
\[
\bar P= C (\tSz_1)^T,
\quad
\bar Q= (\tSzr_1)^T C.
\]
This time a direct computation of $P^{-1}Q$ is necessary. To evaluate $(\tSzr_1)^T$,
we use the fact that the (method of) proof of Theorem \ref{preservesM} shows that
\[
C
\left(
\tQzr_{ \scriptstyle\frac{2n+1}{n+1}-k}
\right)^{-1}
C
=
\tQz_k.
\]
(In fact the proof of Theorem \ref{preservesM} establishes the cases $k=1$ and $k=
1+\scriptstyle\frac1{n+1}$, then the general case follows from cyclic symmetry.)
Writing  $\tSz_1=\tQz_1 \cdots \tQz_{1+\scriptstyle\frac n{n+1}}$, this
leads to $C  \tSzr_1 C= (\tSz_0)^{-1}$, and we obtain
\[
\overline{P^{-1}Q}= (\tSz_1)^{-T} (\tSz_0)^{-T} \ \ (=  (\tSz_1 \tSz_2)^{-T}).
\]
Equation (\ref{tMz^n+1}) gives
$\tSz_1 \tSz_2=A^{n+1}$, which commutes with both $A$ and $B$. This completes the proof.
\end{proof}

\section{Main Results}\label{symp}

In this section we construct a symplectic structure for our
space of monodromy data $\cS_{n+1}^{\text{local}}$. We shall show
that it is in fact a symplectic groupoid, so we begin with some preparations
on this topic.

\subsection{Groupoids.}\label{gpd}

A {\em groupoid} $\cG\rra X$ consists of two sets $\cG$, the set of arrows (called the \ll groupoid\rrr),  
and $X$, the set of units (called the \ll base\rrr), together with
\ll source\rr and \ll target\rr maps
$s,t:\cG\to X$, an object 
inclusion map
$\eps: X\to\cG$, an inverse map $\imath:\cG\to\cG$, and a
\ll partial multiplication map\rr $\m$, subject to a system of axioms.
Roughly speaking, the axioms describe a group whose group operation is only partially defined.

The product $\m(g,h)$ of two elements $g,h\in\cG$ is
defined when  $s(g) =t(h)$. In this case  $g,h$ are said to
be composable, and the set of composable pairs $(g,h)$ is
denoted by $\cG^{(2)}$.
The product $\m(g,h)$ is often denoted by $gh$.
For any $x\in X$, the element $\eps(x)$ is required to act as a unit element
for multiplication,
and is often denoted by $1_x$.

The groupoid is said to be a
{\em Lie groupoid} when $\cG,X,\cG^{(2)}$ are smooth manifolds,
the structure maps $s,t,\eps,\m$ are smooth,  $s,t$ are surjective submersions,
and $\eps$ is an embedding.

To relate two Lie groupoids, there is a notion of Lie groupoid morphisms.

\begin{df}
Let $\cG'$ and $\cG$ be Lie groupoids on $X’$ and $X$ respectively. A Lie groupoid morphism is a pair of smooth maps $F:\cG' \to \cG$, $f:X'\to X$ such that $s\circ F=f\circ s'$, $t\circ F=f\circ t'$ and $F(\fm(h',g'))=\fm(F(h'),F(g'))$, for all $(h',g')\in (\cG')^{(2)}$. Moreover, if $F$ and $f$ are injective immersions, then $\cG'\rightrightarrows X'$ is called a Lie subgroupoid of $\cG\rightrightarrows X$.
\end{df}

The {\em Lie algebroid} of the Lie groupoid is the
normal bundle $A$ to the image of the embedding $\eps$.
If we identify the fibre $A_x$ over $x\in X$
with $T_x s^{-1}(x)$,  then $dt$ gives a map $\rho:A\to TX$, called
the anchor map.
We refer to \cite{Ma01} for the basic theory of Lie groupoids.

Let $Y\subset X$ be a (embedded) submanifold of $X$. Then
$\cG^Y_Y:=s^{-1}(Y)\cap t^{-1}(Y)$ is a groupoid over $Y$ with the source map and target map being the restriction of  $s$ and $t$ on $\cG^Y_Y$ respectively. We have the following
well known
criterion for $\cG^Y_Y \rightrightarrows Y$ to be a Lie groupoid.
\begin{pro}\label{pro:NLiegroupoid}
If
\[ T_{x}Y+\rho(A_x)=T_x X, \quad \mbox{for all }x\in Y,\]
where $\rho:A:=\cup A_x \to TX$ is the anchor map for the associated Lie algebroid, then $\cG^Y_Y:=s^{-1}(Y)\cap t^{-1}(Y)$ is a Lie groupoid over $Y$.
Furthermore, $\cG^Y_Y\rightrightarrows Y$ is a Lie subgroupoid of $\cG\rightrightarrows X$.
\qed
\end{pro}

Next, let us consider 2-forms on Lie groupoids.

\begin{df}
Let $\cG\rightrightarrows X$ be a Lie groupoid and $\omega$ a 2-form on $\cG$. The 2-form $\omega$ is called multiplicative if the graph of the multiplication $\Lambda\subset \cG\times \cG\times \bar{\cG}$ is isotropic, i.e.\ if it satisfies  $\fm^*\omega=\pr_1^*\omega+\pr_2^*\omega$, where $\pr_i$ are the natural projection onto $G$. If $\phi$ be a closed 3-form on $X$ then $\omega$ is called $\phi$-relatively closed if $d\omega=s^*\phi-t^*\phi$.
\end{df}

\begin{df} \cite{Xu04}
A Lie groupoid $\cG\rightrightarrows X$ together with a multiplicative 2-form $\omega\in \Omega^2(\cG)$ and a closed 3-form $\phi\in \Omega^3(X)$ such that $\omega$ is $\phi$-relatively closed is called a pre-quasi-symplectic groupoid. If in addition, one requires that $\dim\cG=2\dim X$ and the intersection $\ker(ds)\cap\ker(dt)\cap\ker(\omega)$ vanishes, then it is called a quasi-symplectic groupoid (or a $\phi$-twisted presymplectic groupoid by \cite{BCWZ}).
\end{df}

\begin{ex} The AMM groupoid $G\times G\rightrightarrows G$ is a quasi-symplectic groupoid (\cite{AMM}, \cite{Xu04}). 
Let $G$ be a complex semisimple Lie group, and let $(.,.)$ be an Ad-invariant non-degenerate symmetric bilinear form on $\fg=\Lie(G)$.
Consider the action groupoid $G\times G\rightrightarrows G$, i.e.\ its structure maps are given by $t(g,a)=gag^{-1}$, $s(g,a)=a$, $\epsilon(a)=(id_G,a)$ and $( (g,a), (h,b) )\in (G\times G)^{(2)}$ if $a=hbh^{-1}$.
Denote the left and right invariant Maurer-Cartan forms on $G$ by $\theta^L$ and $\theta^R$ respectively.
Let $\chi\in \Omega^3(G)$ denote the bi-invariant closed 3 form on $G$ given by
\[\chi=\tfrac{1}{12}(\theta^L,[\theta^L,\theta^L])=\tfrac{1}{12}(\theta^R,[\theta^R,\theta^R]). \]
Consider the $2$-form\footnote{Note that this form matches the 2-form in \cite{AMM} Remark 3.2, and has a sign difference from the 2-form in \cite{Xu04}.} $\omega\in\Omega^2(G\times G)$:
\begin{equation}\label{eq:quasi2form}
\omega_{(g,a)}=\tfrac{1}{2}[(\Ad_a\pr_1^*\theta^L,\pr_1^*\theta^L)+(\pr_1^*\theta^L,\pr^*_2(\theta^L+\theta^R))]
\end{equation}
where $(g,a)\in G\times G$ and $\pr_1,~ \pr_2:G\times G \to G$ are the projections. Direct computation shows that $\omega$ is multiplicative and $\chi$-relatively closed, and that $\ker(ds)\cap\ker(dt)\cap\ker(\omega)$ vanishes. The dimension condition is clearly satisfied, thus it is a quasi-symplectic groupoid (\cite{Xu04}).
\qed
\end{ex}

One can check easily that the properties of $\omega$ being $\phi$-relatively closed and $\omega$ being multiplicative are both preserved under Lie groupoid morphisms:

\begin{lm}\label{lm:multiplicative}
Let $\cG'\rightrightarrows X'$ and $\cG \rightrightarrows X$ be Lie groupoids, and let the pair $F:\cG'\to\cG$, $f:X'\to X$ be a Lie groupoid morphism. If $\omega\in\Omega^2(\cG)$ be a multiplicative $2$-form on $\cG\rightrightarrows X$, then $F^*\omega$ is a mutiplicative $2$-form on $\cG'\rightrightarrows X'$. If $\omega\in \Omega^2(\cG)$ is $\phi$-relatively closed, for some closed $3$-form $\phi\in\Omega^3(X)$, then $F^*\omega\in\Omega^2(\cG')$ is $f^*\phi$-relatively closed.
\qed
\end{lm}

Let us recall some general properties of multiplicative $2$-forms on Lie groupoids that will be useful for us. One can find them in many references, for example, see \cite{MeXX}.

\begin{lm}\label{lm:zero}
Let $\omega\in\Omega^2(\cG)$ be a multiplicative $2$-form on $\cG$. Then the pull-back $\epsilon^*\omega$ to the base $X$ is the zero $2$-form.
\qed
\end{lm}

\begin{pro}\label{pro:nondegenerate}
Let $\cG\rightrightarrows X$ be a Lie groupoid, and let $\omega\in \Omega^2(\cG)$ be a multiplicative $2$-form. If $\ker \omega_p=\{0\}$,  for all $p\in \epsilon(X)$, then $\ker \omega_g=\{0\}$, for all $g\in\cG$. (i.e.\ if $\omega$ is nondegenerate along the units $\epsilon(X)$, then it is nondegenerate on the whole groupoid)
\qed
\end{pro}

\begin{pro}\label{pro:unit_to_all}
Let $\cG\rightrightarrows X$ be a Lie groupoid and $\omega, \omega ' \in \Omega^2(\cG)$ be multiplicative $2$-forms on $\cG$. If $d\omega=d\omega '$ on $\cG$, and $\omega=\omega '$ along $\epsilon(X)$, and if $\cG$ is $s$-connected, then $\omega=\omega '$ on $\cG$.
\qed
\end{pro}
We recall that a groupoid is said to be $s$-connected if all $s$-fibres are connected. Without this assumption, Proposition \ref{pro:unit_to_all} may not hold, see \cite{MeXX} for a counterexample.

We end this section by recalling the definition of a symplectic (Lie) groupoid.
\begin{df}
Let $\cG \rightrightarrows X$ be a Lie groupoid and $\omega$ a 2-form on $\cG$. It is called a symplectic groupoid if $\omega$ is symplectic and multiplicative. If, in addition, the manifolds are complex, the structure maps are holomorphic, and the symplectic form is of type $(2,0)$, then it is called a holomorphic symplectic groupoid \cite{LSX}.
\end{df}

\subsection{The universal centralizer.}\label{doub}
In this section, $G$ is any complex, semisimple, simply connected Lie group.
Recall that the universal centralizer of a Steinberg cross section $\Sigma\subset G$ is $\cZ=\{(g,a)\in G\times \Sigma \mid~ gag^{-1}=a\}\subset G\times G$. 

\begin{pro}\label{pro:universalgroupoid} Let $G$ be a complex, semisimple, and simply connected Lie group. Then the universal centralizer of a Steinberg cross section $\cZ\rightrightarrows \Sigma$ is a holomorphic Lie groupoid. Furthermore it is a Lie subgroupoid of $G\times G\rightrightarrows G$.
\end{pro}

\begin{proof}
To show that it is a Lie subgroupoid, we shall use the criterion of Proposition \ref{pro:NLiegroupoid}. We start with the Lie groupoid $G\times G \rightrightarrows G$. Consider the associated Lie algebroid $(A\rightarrow G,\rho:A\to TG)$. For any $a\in \Sigma$, let $A_a=T_{1_a}s^{-1}(a)$ and $\rho=dt$, then
$\rho(A_a)=\{dt_{1_a}(\beta_a)~\mid~\beta_a \in T_{1_a}s^{-1}(a)\}\subset T_a G$. We claim that $\rho(A_a)=T_a\cC_a$ where $\cC_a$ is the conjugacy class/orbit of $a\in G$. To show this, let $(\gamma(\tau),a) \in s^{-1}(a)$ be a curve with $(\gamma(0),a)=(id_G,a)=1_a$, $(\gamma'(0),0)=\beta_a$, $\tau$ being the parameter of the curve. Then
$dt_{1_x}(\beta_x)=\frac{d}{d\tau}\mid_{\tau=0}\gamma(\tau)x \gamma(\tau)^{-1}$, and $\rho(A_a)=T_a\cC_a$.

Recall that a Steinberg cross section is a submanifold of $G$, and is transversal to all conjugacy orbits of regular elements in G - see \cite{Se11}. Thus we have
\[T_a \Sigma+ \rho(A_a)=T_a\Sigma\oplus T_a\cC_a=T_a G,  \quad \forall ~a\in \Sigma. \]
On the other hand,
$s^{-1}(\Sigma)\cap t^{-1}(\Sigma)=\{(g,a)\in G\times \Sigma\mid~gag^{-1}=a\}$,
because $\Sigma$ is a cross section of conjugacy classes so $gag^{-1}$ must be equal to $a$.
Hence, by Proposition \ref{pro:NLiegroupoid}, $\cZ$ is a Lie groupoid over the Steinberg cross section $\Sigma$,
and their structure maps are just the restrictions, which we shall denote by $s'$, $t'$, $\epsilon'$ and $\fm'$.
It is then clear that, with the natural inclusions $i:\cZ \hookrightarrow G\times G$ and $i_o:\Sigma \hookrightarrow G$, the Lie groupoid $\cZ\rightrightarrows \Sigma$ is a Lie subgroupoid of $G\times G \rightrightarrows G$, and the pair $(i,i_o)$ is the morphism of Lie groupoids. Moreover, For $G$ complex semisimple simply connected, a Steinberg cross section is a complex manifold, and the structure maps are clearly all holomorphic, so $\cZ\rightrightarrows \Sigma$ is a holomorphic groupoid.
\end{proof}

The proof shows that the source map $s'$ and the target map $t'$ of $\cZ\rightrightarrows \Sigma$ are identical. This will become crucial later in proving that $\cZ$ is a holomorphic (algebraic) completely integrable system. 

With respect to the $2$-form of the quasi-symplectic groupoid $G\times G\rightrightarrows G$, we have:

\begin{thm}\label{thm:universalsymplectic}
Let $G$ be a complex, semisimple, and simply connected Lie group. Then the universal centralizer of a Steinberg cross section $\cZ\rightrightarrows \Sigma$ is a holomorphic symplectic groupoid with the holomorphic symplectic form $i^*\omega$, where $i:\cZ\hookrightarrow G\times G$ is the natural inclusion and $\omega$ is defined by equation (\ref{eq:quasi2form}).
\end{thm}
\begin{proof}
By Lemma \ref{lm:multiplicative}, $i^*\omega$ is multiplicative and $i_o^*\chi$-relatively closed, so
$d(i^*\omega)=0$ because $s'=t'$.
For nondegeneracy, by Proposition \ref{pro:nondegenerate}, we only need to show that $i^*\omega$ is nondegenerate along the unit $\epsilon'(\Sigma)$ inside the universal centralizer $\cZ$.

We begin by observing that there is a natural splitting $T_{1_a}\cZ=T_{1_a}\epsilon'(\Sigma)\oplus T_{1_a}((s')^{-1}(a))$ $\forall a\in\Sigma$. Thus for any vector $u\in T_{1_a}\cZ$, we can write $u=u^H+u^F$ where $u^H\in T_{1_a}\epsilon'(\Sigma)$, and $u^F\in T_{1_a} (s')^{-1}(a)$. Hence the calculation of the 2-form is decomposed into
\[(i^*\omega)_{1_a}(u,v)=(i^*\omega)_{1_a}(u^H,v^H)+(i^*\omega)_{1_a}(u^F,v^F)+(i^*\omega)_{1_a}(u^H,v^F)+(i^*\omega)_{1_a}(u^F,v^H),\]
and we calculate them respectively as follows.

(1) $(i^*\omega)_{1_a}(u^H,v^H)=0$, because, by Lemma \ref{lm:zero}, the pull-back to the base of a multiplicative 2-form vanishes.

(2) Consider $u^F,v^F\in T_{1_a}((s')^{-1}(a))\subset T_{1_a}\cZ$. Since $(s')^{-1}(a)=G_a\times \{a\}$, so $i_*(T_{1_a} (s^{-1}(a)))=(T_{id_G}G_a) \times \{0\}$. Then $i_*u^F=(\xi^L(id_G),0)$, $i_*v^F=(\eta^L(id_G),0)$, for some $\xi,~ \eta \in \fg_a\subset \fg$. Hence \begin{eqnarray*}
(i^*\omega)_{1_a}(u^F,v^F)&=&\tfrac{1}{2}\Big(  ( \Ad_a(\theta^L(\xi^L(id_G))), \theta^L(\eta^L(id_G)))
- (\Ad_a(\theta^L(\eta^L(id_G))),\theta^L(\xi^L(id_G))) \\ &&+ (\theta^L(\xi^L(id_G)),0)-(\theta^L(\eta^L(id_G)),0)\Big)=0,
\end{eqnarray*}
where we have used $\Ad_a\xi=\xi=\Ad_{a^{-1}}\xi$ for the last equality because $\xi\in\fg_a$.

(3) Consider $u^H\in T_{1_a}\epsilon'(\Sigma)$ and $v^F\in T_{1_a}((s')^{-1}(a))$.
Then we can write $i_*u^H=(0,\rho^L(a))$, for some $\rho^L(a)\in T_a\Sigma$, $ \rho\in  \fg$, while  $i_*v^F=(\eta^L(id_G),0)$, for some $\eta\in  \fg_a$ as in (2). Thus\[(i^*\omega)_{1_a}(u^H,v^F)= \tfrac{1}{2}\Big( (0,\eta)-(\Ad_a \eta, 0) + (0,0)-(\eta, (\theta^L+\theta^R)(\rho^L(a)))\Big) =-(\eta,\rho),\]
where we used $\theta^R(a)=\Ad_a\theta^L(a)$ as well as $\Ad_a\eta=\eta=\Ad_{a^{-1}}\eta$.

(4) Consider $u^F\in T_{1_a}((s')^{-1}(a))$ and $v^H\in T_{1_a}\epsilon'(\Sigma)$. Similar to (3), $i_*v^H=(0,\varrho^L(a))$, for some $\varrho^L(a)\in T_a\Sigma$, $ \varrho\in  \fg$, and $i_*u^F=(\xi^L(id_G),0)$, for some $\xi\in  \fg_a $ as in (2). Then
\[
(i^*\omega)_{1_a}(u^F,v^H)= \tfrac{1}{2} (\xi+\Ad_{a^{-1}}\xi, \varrho)=(\xi,\varrho).
\]

So for any $u,~v \in T_{1_a}\cZ$, writing them as $i_*u=(0,\rho^L(a))+(\xi^L(id_G),0)$ and $i_*v=(0,\varrho^L(a)) + (\eta^L(id_G),0)$ as above, to show that $(i^*\omega)_{1_a}$ is nondegenerate, we just need to show that for any $u$, we can find $v$ such that $(\xi,\varrho)-(\eta,\rho)\neq 0$. This is indeed the case, because we have $T_a \Sigma \oplus T_a \cC_a= T_aG$ (see the proof of Proposition \ref{pro:universalgroupoid}), and also $T_a \cC_a=T_a G/T_a G_a$. Thus we have $T_a \Sigma\cap (T_a G_a)^{\perp}=\{0\}$, and we are done.

Finally, we want to show that the 2-form is of type $(2,0)$. Recall, a $2$-form $\Omega$ is of type $(2,0)$ if and only if for any real tangent vectors $u$ and $v$, it satisfies $\Omega(Ju, v) = \sqrt{-1} \Omega(u, v)$. Now the complex structure on $\cZ$ is defined by the complex Lie group $G$. From the definition of the $2$-form $i^*\omega$, direct calculation shows that it indeed satisfies this criterion. This completes the proof that $(\cZ\rightarrow \Sigma,i^*\omega)$ is a holomorphic symplectic groupoid.
\end{proof}

For this reason, we denote this holomorphic symplectic $2$-form $i^*\omega$ on $\cZ$ by $\omega_{\bC}$,  and we 
call this holomorphic symplectic groupoid $(\cZ\rightrightarrows \Sigma,\omega_{\bC}) $ the {\em Steinberg groupoid}.

As noted in the introduction, the fact that $\omega_{\bC}$ defines a
symplectic form on the universal centralizer $\cZ$ had
already been established
(\cite{BFM}, \cite{FiTs19}, \cite{Ba22}). The groupoid viewpoint
in Theorem \ref{thm:universalsymplectic} gives a simpler
alternative proof, and it is the one which we shall need for
our main results.

\begin{rem}\label{remarkX.K}
Using similar arguments, one can show that the universal centralizer with respect to the Kostant cross section is also a symplectic groupoid, which may be called the {\em Kostant groupoid}. We postpone a discussion of the relation between these two groupoids.
\qed
\end{rem}

\begin{rem}\label{remarkX.B}
Another example is the symplectic groupoid introduced by Bondal in \cite{Bo04}.
Let $\cT_{n+1}$ be the subspace of $\GL_{n+1}\bC$
consisting of upper triangular matrices with all diagonal entries equal to $1$ (\ll Stokes matrices\rrr).  Let
$\cG_{n+1}= \{ (B,A) \in \GL_{n+1}\bC \times \cT_{n+1} \st B^{-T} A B^{-1} \in \cT_{n+1} \}$.
A groupoid $\cG_{n+1}\rra\cT_{n+1}$ is obtained by taking $s(B,A)=A$,
$t(B,A)=B^{-T} A B^{-1}$, $\epsilon(A)=(I,A)$, and
$\m( (B_1,A_1),(B_2,A_2) )= (B_1B_2,A_2)$ when $A_1=B_2^{-T} A_2 B_2^{-1}$.

Evidently the unipotent subgroup $\cT_{n+1}$ here may be replaced by any other maximal unipotent subgroup.  Let us use the one
given by the positive roots
$\cRz_1\cup \dots\cup \cRz_{1+\scriptstyle\frac n{n+1}}$
(see the remarks following Proposition \ref{scriptM}).
Then it can be verified that map
\[
\cS_{n+1}^{\text{local}}\to \cG_{n+1},\quad
(\tE,\tMz)\mapsto (\tE, (\tSz_1)^{-T})
\]
is an injective morphism of groupoids. In particular, the matrix $\tMz$ in our theory
plays the same role as the Stokes matrix in Bondal's theory.
\qed
\end{rem}

It is known (see \cite{We87}) that the base of any symplectic groupoid
inherits a Poisson structure and the source map is Poisson while the target map is anti-Poisson.
Since $s=t$ for the universal centralizer, the Poisson structure on the base is trivial. 
Moreover, all the $s$-fibres are Lagrangian submanifolds of the same dimension.

A completely integrable Hamiltonian system is a symplectic manifold which has the maximal number of Poisson commuting independent integrals of motion (i.e.\ one half the dimension of the manifold). To say that the functions are independent  means that their differentials are linearly independent on an open dense subset. The universal centralizer of a Steinberg cross section turns out to be a (holomorphic, in fact algebraic) example of this, and we can view it (in our context) as a group analogue of Hitchin’s completely integrable system.

\begin{thm}
The Steinberg groupoid $\cZ\rightrightarrows \Sigma$ is a holomorphic (in fact, algebraic) completely integrable Hamiltonian system. 
In particular, $\cZ$ is a (singular) Lagrangian fibration over $\Sigma$. The generic fibres are complex algebraic tori and the singular fibres are those over non-semisimple elements of $\Sigma$.
\end{thm}

\begin{proof}
We have already seen that the groupoid $\cZ\rightrightarrows\Sigma$ is a Lagrangian (singular) foliation because the s-fibers are Lagrangian submanifolds in $\cZ$. 
Since the Poisson structure is trivial on $\Sigma$, any functions on $\Sigma$ Poisson commute, and since the source map $s$ is Poisson, any functions pulled back to $\cZ$ by $s$ also Poisson commute.
Thus, in order to show that this is an integrable system, we only need to find $\frac{1}{2}\dim \cZ$ independent functions on $\cZ$.

Let $\chi_i:~G\to \bC~ (1\leq i \leq \ell)$ denote the fundamental characters of $G$. 
Steinberg showed that if $G$ is a semisimple simply connected (algebraic) group, an element $a\in G$ is regular if and only if the differentials of $\chi_1,\cdots,\chi_\ell$ are linearly independent at $a$ (Theorem 1.5 of \cite{St65}).
Thus, the $\chi_i$ are independent on $\Sigma$. In fact, Steinberg used the map $\chi=(\chi_1,\cdots,\chi_\ell)$ to define the isomorphism from $\Sigma$ to $\bC^\ell$. 
Since $s$ is a submersion (it is the projection map to the second component), if functions on $\Sigma$ are independent, their pull-backs under $s$ are independent on $\cZ$. 
Thus we obtain a set of $\ell$ independent functions on $\cZ$ which pairwise Poisson commute, and $\dim\cZ=2\ell$.

The $s$-fibre of $a\in \Sigma$ is the stabilizer $G_a$, which will be an $\ell$-dimensional complex algebraic torus if $a$ is semisimple. 
\end{proof}

\begin{remark}
We have verified explicitly in the cases $G=\SL_2\bC$ and $\SL_3\bC$, that, when $g$ and $a$ are both semisimple, then (the logarithms of) their eigenvalues give action-angle coordinates.
\end{remark}

\subsection{Anti-symmetry and Reality.}\label{antisymm}

Recall from Section \ref{conn} that both $\sigma:\cZ\to\cZ$ and $\theta:\cZ\to\cZ$ are involutions, and that $\sigma$ commutes with $\theta$. In fact, they are also compatible with the groupoid structure. To emphasise the dependence of the involutions on the Stokes data $M$, we shall use $(E,M)$ to denote the points in $\cZ$ throughout this section.
Since we shall not be using the AMM groupoid $G\times G\rightrightarrows G$ anymore, we simply use the notation $s,\ t,\ \epsilon, \fm$ to represent the corresponding maps for the universal centralizer groupoid $\cZ\rightrightarrows \Sigma$ and drop the $'$ notations that were used up to this point.

\begin{pro}\label{pro:commute}
The pairs of maps $(\sigma,\sigma_o)$ and $(\theta,\theta_o)$ are both (involutory) Lie groupoid morphisms from $\cZ\rightrightarrows \Sigma$ onto itself. Furthermore, these two Lie groupoid morphisms commute with each other.
\end{pro}

\begin{proof}
From the definition of $\sigma$ and $\sigma_o$ , clearly $\sigma_o\circ s=s\circ\sigma$, so we only need to show
$$\sigma( \fm((E_1,M_1),(E_2,M_2)))=\fm(\sigma(E_1,M_1),\sigma(E_2,M_2)),\quad \forall((E_1,M_1),(E_2,M_2))\in \cZ^{(2)}.$$
Given any $((E_1,M_1),(E_2,M_2))\in \cZ\times \cZ$, this is in $\cZ^{(2)}$ if and only if $M_1=E_2M_2E_2^{-1}$. 
As $M_1$ and $M_2$ are in $\Sigma$, a Steinberg cross section, we have $M_1=M_2$ and
$\fm((E_1,M_1),(E_2,M_2))=(E_1E_2,M_2)$. Thus
\begin{eqnarray*}
\fm(\sigma(E_1,M),\sigma(E_2,M_2))&=& (\Ad_{F_\sigma(M)} E_1^{-T} E_2^{-T}, \Ad_{F_\sigma(M)} M^{-T} )  \\
&=& \sigma(E_1E_2,M)=\sigma(\fm((E_1,M),(E_2,M)) ).
\end{eqnarray*}
Thus $(\sigma,\sigma_o)$ is a Lie groupoid morphism from $\cZ\rightrightarrows \Sigma$ onto itself.
The same argument shows that $(\theta,\theta_o)$ is a Lie groupoid morphism. By Theorem \ref{comminv}, they are commuting Lie groupoid morphisms.
\end{proof}

This gives us:
\begin{cor}\label{cor:fixedpoint_groupoid}
The fixed points sets $ \cZ^{\sigma} \rightrightarrows \Sigma^{\sigma_o} $, $ \cZ^{\theta} \rightrightarrows \Sigma^{\theta_o} $  are Lie groupoids. Moreover, they are both Lie subgroupoids of the universal centralizer groupoid $\cZ\rightrightarrows \Sigma$ with the natural inclusions $j_1:\cZ^{\sigma}\hookrightarrow \cZ$, $(j_1)_0:\Sigma^{\sigma_o}\hookrightarrow \Sigma$, and $j_2:\cZ^{\theta}\hookrightarrow \cZ$, $(j_2)_0:\Sigma^{\theta_o}\hookrightarrow \Sigma$.
\end{cor}
\begin{proof}
The fixed point sets $\cZ^{\sigma}$ and $\Sigma^{\sigma_o}$ are smooth submanifolds because $\sigma$ and $\sigma_o$ are both (isometric) involutions. The restriction $s|_{\cZ^\sigma}:\cZ^\sigma\rightrightarrows \Sigma^{\sigma_o}$  is surjective, because $(1,M)\in\cZ^\sigma$ for all $M\in\Sigma^{\sigma_o}$ and $s(1,M)=M$. Moreover, since $s$ is a surjective submersions, and $\sigma$, $\sigma_o$ are involutions, the restriction $s|_{\cZ^\sigma}$ is a submersions also. Thus $s|_{\cZ^\sigma}\to\Sigma^{\sigma_o}$ is a surjective submersion. The map $t|_{\cZ^\sigma}$ is also a surjective submersion because $t=s$. The same argument works for $(\theta,\theta_0)$.
\end{proof}

In general, we have:
\begin{lm}\label{lm:commute}
Given a Lie groupoid $\cG\rightrightarrows X$ and two commuting Lie groupoid morphisms $(\sigma,\sigma_o)$ $(\theta,\theta_o)$ onto itself, which are both involutions, and such that $\cG^\sigma\rightrightarrows X^{\sigma_0}$ is a Lie groupoid, then $(\theta|_{\cG^\sigma},\theta_0|_{X^{\sigma_0}})$ is a Lie groupoid morphism from $\cG^\sigma\rightrightarrows X^{\sigma_0}$ onto itself and $\theta|_{\cG^\sigma}$, $\theta_0|_{X^{\sigma_0}}$ are both involutions.
\end{lm}
\begin{proof}
Since $\theta\circ\sigma=\sigma\circ\theta$ so $\theta(\cG^\sigma)\subset \cG^\sigma$ and $\theta|_{\cG^\sigma}$ is an involution on $\cG^\sigma$ because $\theta$ is an involution on $\cG$. Similarly, since $\theta_0\circ\sigma_0=\sigma_0\circ\theta_0$, so $\theta_0(X^{\sigma_0})\subset X^{\sigma_0}$ and $\theta_0|_{X^{\sigma_0}}$ is an involution on $X^{\sigma_0}$. 
\end{proof}

Next, combining Proposition \ref{pro:commute}, Lemma \ref{lm:commute}, and Corollary \ref{cor:fixedpoint_groupoid}, we obtain:
\begin{cor}\label{cor:desired_groupoid}
$ (\cZ^{\sigma})^\theta \rightrightarrows (\Sigma^{\sigma_0})^{\theta_0} $ is a Lie groupoid and a Lie subgroupoid of the universal centralizer groupoid with the natural inclusions $k: (\cZ^{\sigma})^\theta\hookrightarrow \cZ$ and  $k_0: (\Sigma^{\sigma})^{\theta}\hookrightarrow \Sigma$.
\qed
\end{cor}

Before stating the main results, let us fix some notations. Recall that $\omega_\bC=\omega_1+\sqrt{-1}\omega_2$ is a holomorphic symplectic form, where $\omega_1$ and $\omega_2$ are real valued symplectic forms on $\cZ$.
Denote $\overline{\omega}_\bC:=\omega_1-\sqrt{-1}\omega_2$. 
The key result of this section is the following:

\begin{thm}\label{thm:main}
(i) $\sigma^*\omega_\bC=\omega_\bC$ along the units $\epsilon(\Sigma)$. (ii) $\theta^*\omega_{\bC}=-\overline{\omega}_\bC$ along the units $\epsilon(\Sigma)$.
\end{thm}

\begin{proof}
We shall use a similar idea to that in the proof of Theorem \ref{thm:universalsymplectic}. Recall the natural splitting of the tangent space $T_{(id_G,M)}\cZ=T_{(id_G,M)}\epsilon(\Sigma)\oplus T_{(id_G,M)}(s^{-1}(M))$ along the units, and write $u=u^H+u^F$ for the associated decomposition of $u\in T_{(id_G,M)}\cZ$.

We begin by expressing the tangent vectors as velocity vectors of suitable curves.
A tangent vector $u\in T_{(id_G,M)}\cZ$ can be expressed by
$$u=\tfrac{d}{dt}|_{t=0}(r_E(t),L_M(r_M(t)))\in T_{(id_G,M)}\cZ,$$
where $r_E(t)$ and $r_M(t)$ are curves in $G$ with $r_E(0)=id_G=r_M(0)$, such that the curve $L_M(r_M(t))\in\Sigma$ and $r_E(t)\in G$, which satisfy the commuting property for some neighborhood of $0$. Here $L_M$ denotes the left group multiplication by $M$. Hence the tangent vector
$$u=(r'_E(0),(L_M)_*(r'_M(0)) )=( (r'_E(0))^L(id_G), (r'_M(0))^L(M) ),$$
has the decomposition $u^H=\frac{d}{dt}|_{t=0}(id_G,L_M(r_M(t)))=(0,(r'_M(0))^L(M))$ and $u^F=u-u^H=( (r'_E(0))^L(id_G),0 )$. We shall express $u^F=\frac{d}{dt}|_{t=0}(\tilde{r}_E(t),M)$ for some $\tilde{r}_E(t)\in G_M$, where $\tilde{r}_E(0)=id_G$, $\tilde{r}'_E(0)=r'_E(0) \in \fg_M$.
Here the notation $\xi^L(g)$ means the left invariant vector field of $\xi\in\fg$ evaluated at $g\in G$.

Similarly, we express another tangent vector $v$ by
$$v=\tfrac{d}{dt}|_{t=0}(\delta_E(t),L_M(\delta_M(t)))\in T_{(id_G,M)}\cZ,$$ with the same condition.
Then $v^H=\frac{d}{dt}|_{t=0}(id_G,L_M(\delta_M(t)))=(0,(\delta'_M(0))^L(M))$ and $v^F=v-v^H=( (\delta'_E(0))^L(id_G),0 )$. We shall express $v^F=\frac{d}{dt}|_{t=0}(\tilde{\delta}_E(t),M)$ for some $\tilde{\delta}_E(t)\in G_M$, where $\tilde{\delta}_E(0)=id_G$, $\tilde{\delta}'_E(0)=\delta'_E(0) \in \fg_M$.

To simplify notation, the holomorphic symplectic form $\omega_\bC$ will be denoted by $\omega$ for the rest of the proof. Recall from the calculation in the proof of Theorem \ref{thm:universalsymplectic} that
$$\omega_{(id_G,M)}(u,v)=\omega_{(id_G,M)}(u^H,v^F)+\omega_{(id_G,M)}(u^F,v^H)$$
because $\omega_{(id_G,M)}(u^H,v^H)=0$ and $\omega_{(id_G,M)}(u^F,v^F)=0$. Thus we shall only need to compute the cross terms.

{\bf (i) The map $\sigma$}

Since $(\sigma,\sigma_0)$ is a groupoid morphism, direct calculation shows that $(\sigma_*u)^H=\sigma_*(u^H)$ and $(\sigma_*u)^F=\sigma_*(u^F)$. 
Thus $\omega_{\sigma(id_G,M)}(\sigma_*u,\sigma_*v)=\omega_{\sigma(id_G,M)}(\sigma_*(u^H),\sigma_*(v^F))
+\omega_{\sigma(id_G,M)}(\sigma_*(u^F),\sigma_*(v^H)),$  and  we only need to find the formula for $\sigma_*(u^H)$ and $\sigma_*(u^F)$.

Denote the tangent vectors by $\xi_1=\tilde{r}_E'(0),~\eta_1=\tilde{\delta}_E'(0)\in \fg_M$, and $\xi_2=r_M'(0), ~\eta_2=\delta_M'(0)\in\fg$, so $u=( \xi_1^L(id_G),\xi_2^L(M) )$ and $v=(\eta_1^L(id_G),\eta^L_2(M))$.
Then using the curves $\tilde{r}_E(t)$ and $\tilde{\delta}_E(t)$, we obtain
\begin{eqnarray*}
&&\sigma_*(u^F)=\tfrac{d}{dt}|_{t=0} \left( F_\sigma(M)(\tilde{r}_E(t)^{-T})F_\sigma(M)^{-1},  \sigma_o(M) \right)=\left( (-\Ad_{F_\sigma(M)} \xi_1^T)^L (id_G), 0 \right) \\
&&\sigma_*(v^F)=\tfrac{d}{dt}|_{t=0} \left( F_\sigma(M)(\tilde{\delta}_E(t)^{-T})F_\sigma(M)^{-1},  \sigma_o(M) \right)=\left( (-\Ad_{F_\sigma(M)}\eta_1^T)^L(id_G),0 \right).
\end{eqnarray*}
Next, $\sigma_*(u^H)=\frac{d}{dt}|_{t=0} (id_G, F_\sigma(L_M(r_M(t))) (L_M(r_M(t)))^{-T} F_\sigma(L_M(r_M(t)))^{-1} \in T_{(id_G,\sigma_o(M))}\epsilon(\Sigma)$, so
\begin{eqnarray*}
&&\sigma_*(u^H)=\left( 0, \left(\Ad_{F_\sigma(M) M^T F_\sigma(M)^{-1}}( \tfrac{d}{dt}|_{t=0}F_\sigma(L_M(r_M(t)))F_\sigma(M)^{-1}) \right)^L (\sigma_o(M)) \right) \\&&
+ \left( 0, \left( -\Ad_{F_\sigma(M)}\xi_2^T \right)^L(\sigma_o(M) ) \right) + \left( 0,\left( -\tfrac{d}{dt}|_{t=0}F_\sigma(L_M(r_M(t)))F_\sigma(M)^{-1} \right)^L( \sigma_o(M) ) \right),
\end{eqnarray*}
and
\begin{eqnarray*}
&&\sigma_*(v^H)=\left( 0, \left(\Ad_{F_\sigma(M)M^TF_\sigma(M)^{-1}}( \tfrac{d}{dt}|_{t=0}F_\sigma(L_M(\delta_M(t)))F_\sigma(M)^{-1}) \right)^L (\sigma_o(M) ) \right) \\&&
+ \left( 0, \left(-\Ad_{F_\sigma(M)}\eta_2^T\right)^L(\sigma_o(M) )\right) + \left( 0, \left(-\tfrac{d}{dt}|_{t=0}F_\sigma(L_M(\delta_M(t)))F_\sigma(M)^{-1} \right)^L( \sigma_o(M) ) \right).
\end{eqnarray*}

Calculating the value of the symplectic form with these curves, we obtain
\begin{eqnarray*}
\lefteqn{2 \omega_{(id_G,\sigma_o(M))}(\sigma_*(u^H),\sigma_*(v^F)) }\\
&&=\left( \eta_1^T,-\xi_2^T-\Ad_{M^{-T}}(\xi_2^T) \right)+\left( \Ad_{M^{-T}} (\eta_1^T) -\Ad_{M^T} (\eta_1^T), \Ad_{F_\sigma(M)^{-1}}( \tfrac{d}{dt}|_{t=0} F_\sigma(L_Mr_M(t))F_\sigma(M)^{-1}) \right)
\end{eqnarray*}
Since $\eta_1\in\fg_M$, $\Ad_M\eta_1=\eta_1=\Ad_{M^{-1}}\eta_1$ and $\Ad_{M^{-T}}(\eta_1^T)=\eta_1^T=\Ad_{M^T}\eta^T_1$. 
Hence the second term in the last equation vanishes, and we have obtained
\begin{equation}\label{eq:sigma_uHvF}
2\omega_{(id_G,\sigma_o(M) )}(\sigma_*(u^H),\sigma_*(v^F))=(\eta_1^T,-\xi_2^T-\Ad_{M^{-T}}(\xi_2^T) )=2(\eta_1^T, -\xi_2^T)
\end{equation}

A similar calculation gives 
\begin{equation}\label{eq:sigma_uFvH}
2\omega_{(id_G, \sigma_o(M) )}(\sigma_*(u^F),\sigma_*(v^H))=-(\xi_1^T,-\eta_2^T-\Ad_{M^{-T}}(\eta_2^T) )=-2(\xi_1^T,-\eta_2^T)
\end{equation}

On the other hand, we know from the calculation in the proof of Theorem \ref{thm:universalsymplectic} that
\begin{equation}\label{eq:uHvF}
2\omega_{(id_G,M)}(u^H,v^F)=(-\eta_1,\xi_2+\Ad_M\xi_2)=2(-\eta_1,\ \xi_2)
\end{equation}
and
\begin{equation}\label{eq:uFvH}
2\omega_{(id_G,M)}(u^F,v^H)=(\xi_1,\eta_2+\Ad_M\eta_2)=2(\xi_1,\ \eta_2)
\end{equation}
Since $(X,Y)=\Tr(XY)$ for $\SL(n,\bC)$, we have
$$(\eta_1^T,-\xi_2^T )=-\Tr(\xi_2\eta_1)=-(\eta_1,\xi_2), \quad\mbox{and}\quad (-\xi_1^T,-\eta_2^T)=\Tr(\xi_1\eta_2)=(\xi_1,\eta_2),$$
so equation (\ref{eq:sigma_uHvF}) is equation (\ref{eq:uHvF}), and equation (\ref{eq:sigma_uFvH}) is equation (\ref{eq:uFvH}). 
Since $u$, $v$ are arbitrary, we conclude that $\sigma^*\omega=\omega$ along $\epsilon(\Sigma)$ in $\cZ$.

{\bf (ii) The map $\theta$}

For the same reason as before, we only need to calculate $\theta_*(u^H)$ and $\theta_*(u^F)$.
We have
\[(\theta^*\omega)_{(id_G,M)}(u,v)
=\omega_{\theta(id_G,M)}(\theta_*(u^H),\theta_*(v^F))+\omega_{\theta(id_G,M)}(\theta_*(u^F),\theta_*(v^H)).
\]

Using the same description of the tangent vector $u$ and $v$ and arguing as in (i), we have
\[\theta_*(u^F)=( (\Ad_{F_\theta(M)} \overline{\xi}_1)^L (id_G), 0 )\quad \mbox{and} \quad
\theta_*(v^F)=( (\Ad_{F_\theta(M)}\overline{\eta}_1)^L(id_G),0 ).\]
for the vertical part.

For the horizontal part we have:
\begin{eqnarray*}
&&\theta_*(u^H)=\left( 0, \left( \Ad_{F_\theta(M) \overline{M} F_\theta(M)^{-1}}( \tfrac{d}{dt}|_{t=0}F_\theta(L_M(r_M(t)))F_\theta(M)^{-1}) \right)^L (\theta_o(M) ) \right) \\&&
+ \left( 0, \left( -\Ad_{F_\theta(M)}\overline{\xi}_2 \right)^L(\theta_o(M) ) \right) + \left( 0,\left( -\tfrac{d}{dt}|_{t=0}F_\theta(L_M(r_M(t)))F_\theta(M)^{-1} \right)^L( \theta_o(M) ) \right),
\end{eqnarray*}
and
\begin{eqnarray*}
&&\theta_*(v^H)=\left( 0, \left( \Ad_{F_\theta(M) \overline{M} F_\theta(M)^{-1}}( \tfrac{d}{dt}|_{t=0}F_\theta(L_M(\delta_M(t)))F_\theta(M)^{-1}) \right)^L (\theta_o(M) ) \right) \\&&
+ \left( 0, \left( -\Ad_{F_\theta(M)}\overline{\eta}_2 \right)^L(\theta_o(M) ) \right) + \left( 0,\left( -\tfrac{d}{dt}|_{t=0}F_\theta(L_M(\delta_M(t)))F_\theta(M)^{-1} \right)^L( \theta_o(M) ) \right).
\end{eqnarray*}

Substituting into $\omega$ we have:
\begin{eqnarray*}
\lefteqn{2\omega_{(id_G,\theta_o(M))}(\theta_*(u^H),\theta_*(v^F)) }\\
&&= \left( -\overline{\eta}_1,-\overline{\xi}_2-\Ad_{\overline{M}^{-1}}\overline{\xi}_2 \right)
 + \left( \Ad_{\overline{M}}\overline{\eta}_1-\Ad_{\overline{M}^{-1}}\overline{\eta}_1,\Ad_{F_\theta(M)^{-1}}\tfrac{d}{dt}|_{t=0}F_\theta(L_M(r_M(t)))B(M)^{-1} \right)
\end{eqnarray*}

Since $\eta_1\in\fg_M$, the first term and the third term of the last equation cancel out. Thus we obtain
$$2\omega_{(id_G,\theta_o(M))}(\theta_*(u^H),\theta_*(v^F))=(\overline{\eta}_1,\overline{\xi}_2+\Ad_{\overline{M}^{-1}}\overline{\xi}_2)=2(\overline{\eta}_1,\overline{\xi}_2)$$

A similar calculation gives us
$$2\omega_{(id_G,\theta_o(M))}(\theta_*(u^F),\theta_*(v^H))=-2\omega_{(id_G,\theta_o(M))}(\theta_*(v^H),\theta_*(u^F))=-2(\overline{\xi}_1,\overline{\eta}_2).$$
So $\omega_{(id_G,\theta_o(M))}(\theta_*(u),\theta_*(v))=(\overline{\eta}_1,\overline{\xi}_2)-(\overline{\xi}_1,\overline{\eta}_2)\in\bC$,
comparing to $\omega_{(id_G,M)}(u,v)=(\xi_1,\eta_2)-(\eta_1,\xi_2) \in\bC$, we see that $$\omega_{(id_G,\theta_o(M)}(\theta_*(u),\theta_*(v))=-\overline{\omega_{(id_G,M)}(u,v)}. $$
We conclude that $\theta^*\omega=-\overline{\omega}$ along $\epsilon(\Sigma)$ in $\cZ$.
\end{proof}

Since not all s-fibres are not connected in our case, we shall need the following modification of Proposition \ref{pro:unit_to_all}:
\begin{pro}\label{pro:dense}
Given a Lie groupoid $\cG\rightrightarrows X$.  Let $\omega$ and $\omega'$ be two multiplicative 2-forms on $\cG$. Suppose that $d\omega=d\omega'$ on $\cG$ and $\omega=\omega'$ along the unit $\epsilon(X)$. If the set
$$\check{X}=\{x\in X ~|\ s^{-1}(x) \mbox{ is connected} \} $$ is an open dense subset of $X$, then $\omega=\omega'$ on the whole groupoid $\cG$.
\qed
\end{pro}

\begin{cor}\label{cor:main}
(i) $\sigma:\cZ \to \cZ$ is a holomorphic symplectomorphism, i.e.\ $\sigma^*(\omega_\bC)=\omega_\bC$ on $\cZ$ (ii) $\theta^*(\omega_\bC)=-\overline{\omega}_\bC$ on $\cZ$.
\end{cor}

\begin{proof}
From Theorem \ref{thm:main} we know that  $\sigma^*\omega_\bC=\omega_\bC$ and $\theta^*\omega_\bC=-\overline{\omega}_\bC$ along the units $\epsilon(\Sigma)$. From Lemma \ref{lm:multiplicative}, we know that both $\sigma^*\omega_\bC$ and $\theta^*\omega_\bC$ are multiplicative, and $d(\sigma^*\omega_\bC)=d\omega_\bC$ and $d(\theta^*\omega_\bC)=0=d\omega_\bC$ on $\cZ$.  For $G=\SL_{n+1}\bC$, if $a\in G$ is semisimple, then the stabilizer $G_a$ is connected (in fact a complex algebraic torus), and the set of semisimple elements forms an open dense subset of the set of regular elements. Now apply Proposition \ref{pro:dense} to obtain $\sigma^*(\omega_\bC)=\omega_\bC$ on $\cZ$, and $\theta^*(\omega_\bC)=-\overline{\omega}_\bC$ on $\cZ$.
\end{proof}

We end this section with the following conclusion:

\begin{thm}\label{thm:result}
$ (\cZ^{\sigma} \rightrightarrows \Sigma^{\sigma_o}, j_1^*(\omega_\bC))$ is a holomorphic symplectic groupoid, while  
$ (\cZ^{\theta} \rightrightarrows \Sigma^{\theta_o}, j_2^*(\omega_2)) $ and $( (\cZ^{\sigma})^\theta \rightrightarrows (\Sigma^{\sigma_o})^{\theta_o}, k^*(\omega_2))$ are real symplectic groupoids.
\end{thm}
\begin{proof}
We have shown in Corollary \ref{cor:fixedpoint_groupoid} and Corollary \ref{cor:desired_groupoid} that $ \cZ^{\sigma} \rightrightarrows \Sigma^{\sigma_o}$, $ \cZ^{\theta} \rightrightarrows \Sigma^{\theta_o}$, and $(\cZ^{\sigma})^\theta \rightrightarrows (\Sigma^{\sigma_o})^{\theta_o}$ are all Lie groupoids. It remains to show that the pulled back 2-forms are symplectic and multiplicative. By Lemma \ref{lm:multiplicative}, $j_1^*\omega_\bC$, $j_2^*\omega_\bC$ and $k^*\omega_\bC$ are multiplicative and closed because $(j_1,(j_1)_o)$, $(j_2,(j_2))_o)$, and $(k,k_o)$ are all Lie subgroupoid morphisms. The 2-form $j_1^*\omega_\bC$ is holomorphic sympletic because $\sigma$ is a holomorphic symplectomorphism. Since $\omega_\bC=\omega_1+\sqrt{-1}\omega_2$, $\theta^*\omega_\bC=-\overline{\omega}_\bC$ implies $j_2^*\omega_1=0$ and $j_2^*\omega_2=-\sqrt{-1}j_2^*\omega_\bC$ is a real symplectic form on $\cZ^\theta$. Since the groupoid morphisims $\theta$ and $\sigma$ commute,  $(\theta|_{\cZ^\sigma})^*(j_1^*\omega_\bC)=-\overline{j_1^*\omega_\bC}$ and since the inclusion $k$ can be viewed as the composition map $(j_1|_{\cZ^\sigma})\circ (j_2|_{(\cZ^\sigma)^\theta})$, we have that $k^*\omega_2=-\sqrt{-1}k^*\omega_\bC$ is a real symplectic form on $(\cZ^\theta)^\sigma$.
\end{proof}

\begin{cor}\label{cor:result}
The space of all local solutions $S^{\text{local}}_{n+1}$ is a real symplectic groupoid over the Stokes data $\cM_{n+1}^{\text{local}}$. 
\qed
\end{cor}

\end{document}